\documentclass[a4paper,12pt]{book}

  \usepackage{amssymb}
  \usepackage{latexsym}
  \usepackage{amsfonts}
  \usepackage{amsthm}
  \usepackage{amsmath}

\usepackage{graphicx}              
\usepackage{framed}
\usepackage{verbatim}
\usepackage{url}

  \newtheorem{theorem}{Theorem}[chapter]
  \newtheorem{lemma}[theorem]{Lemma}

    \newtheorem{problem}{Problem}

  \newtheorem{conj}{Conjecture}

  \theoremstyle{definition}

  \newcommand{\blanknonumber}{\newpage\thispagestyle{empty}}

  \usepackage{graphics}
  \usepackage{graphicx}
  \usepackage{url}

\begin{document}

  \frontmatter



\begin{center}
\thispagestyle{empty}
\vspace*{\fill} \Huge
                        \textsc{Explicit Estimates \\ in the Theory of \\ Prime Numbers}
\\
\vfill\vfill\Large
                          Adrian William Dudek
\\
\vfill\vfill
                          September 2016
\\
\vfill\vfill \normalsize
         A thesis submitted for the degree of Doctor of Philosophy\\
         of the Australian National University
\vfill
         \includegraphics[scale=0.24]{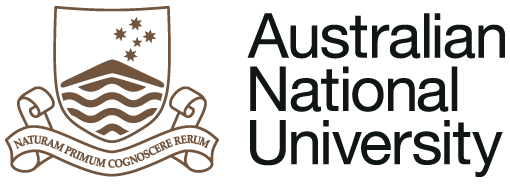}
         
         \vfill
         
      \copyright \ Copyright by Adrian William Dudek 2016 \\ All Rights Reserved

\end{center}

\blanknonumber\ \blanknonumber

\vspace*{\fill}

\begin{center}\emph{
For Sarah and Theodor
}
\end{center}

\vfill\vfill\vfill

\newpage 

\hspace{1in} \thispagestyle{empty}

\newpage

\hspace{1in} \vspace{2in}
\thispagestyle{empty}
\begin{center} \emph{
If you peek and pry where the prime numbers lie,\\
you had better prepare for confusion, \\
for of those who have tried, just a few are inside \\
of a world so unfit for intrusion.\\
\hspace{1in} \\
You see, prime numbers play in a curious way, \\
with a fashion set far from your fielding,\\
but you're welcome to learn through the trials of your turn, \\
any tool that you wield is unyielding.\\
\hspace{1in} \\
Thus with wonder we tread, through conjectures unfed, \\
of a market whose merchants seem yonder,\\
but each rule I have found thus enforces the bound\\
that we only see what we ponder.\\
}
\end{center}

\thispagestyle{empty}

  \setlength{\parskip}{16pt}
    \setlength{\parindent}{0em}

\chapter*{Declaration}\label{declaration}
\thispagestyle{empty}
I hereby declare that the material in this thesis is my own original work except where otherwise stated.

The material in Chapter 2 and in Section 3.1 has been accepted for publication~\cite{dudekcubes} in \textit{Functiones et Approximatio, Commentarii Mathematici}.

The material in Section 3.2 has been published \cite{dudekcramer} in the \textit{International Journal of Number Theory}.

The material in Section 4.1 has been published \cite{dudekestermann} in the \textit{Ramanujan Journal}. 

The material in Section 4.2 is joint work with Dave Platt of the University of Bristol, and has been published \cite{dudekplatt2} in the \textit{LMS Journal of Computation and Mathematics}.

The material in Chapter 5 is again joint work with Dave Platt and has been published \cite{dudekplattem} in \textit{Experimental Mathematics}.

\vspace{1in}

\hfill\hfill\hfill
Adrian W. Dudek
\hspace*{\fill}

  \setlength{\parskip}{16pt}

\chapter*{Acknowledgements}\label{acknowledgements}
\addcontentsline{toc}{chapter}{Acknowledgements}
  \setlength{\parindent}{0em}

It is no small thing, the gratitude I have for my supervisor Dr. Tim Trudgian. In fact, a significant effort was required to keep this thesis from becoming a treatise on his brilliance as a mentor. Before I had even started my PhD, Tim had quite the hold on my email inbox, providing me with judicious nuggets of advice and banter of the highest quality. When I became his student, he offered me the opportunity of weekly one hour meetings, and I can not stress the central part that this played in my candidature. I would always leave our meetings feeling grounded, determined and with plenty to chew on in the lead-up to the next bout of Titchmarsh-and-tea.

Importantly, under Tim's invariably-tweed wing, I was able to learn all about analytic number theory, and whilst this is seemingly displaced from the aspirations of most, this was all I wanted to do when I was twenty one. Tim also talked about cricket quite a lot, though none of this seems to have stuck.

A warm and very gracious thanks must go to my collaborators Lo\"{\i}c Greni\'{e}, Giuseppe Molteni, \L ukasz Pa\'nkowski, Dave Platt and Victor Scharaschkin. It was always a treat to wake up to an inbox full of clever ideas and well wishes, and it was a welcome challenge for me to try and match each of you in our joint work. May our collaborations continue well into the future.

I would also like to acknowledge the Australian Government for the financial support provided by the Australian Postgraduate Award, Dr. Tim Trudgian (again) for affording to me an ANU Supplementary Scholarship, the Science and Industry Endowment Fund for sponsoring my flight to Heidelberg, the International Alliance of Research Universities for funding my trip to the UK, the Australian Mathematical Society for financial support on several occasions to attend their yearly meeting, the Australian Mathematical Sciences Institute for funding my participation in both their Winter School and the Number Theory Down Under Meeting, and the University of Milan for supporting me as a visiting researcher. All of these contributions have been significant in making my candidature both productive and enjoyable.

The Mathematical Sciences Institute at the Australian National University has been a terrific fort, providing me with a stimulating atmosphere, a peaceful office, a wealth of knowledgeable colleagues and a travel fund. Moreover, its location allowed for many leafy ambles and rounds of lunchtime soccer.

On a heartfelt note, my wife-to-be Sarah allowed herself to be neatly packed up and brought from Perth to Canberra for three years just so I could learn about prime numbers. Surpassing greatly, however, her innate ability to be well-bundled and transported, has been her love and support, challenging but considerate advice, happiness during hard times and, quite importantly, her grounded non-mathematician perspective.

Finally, in advance, I would like to acknowledge and thank the examiners of this thesis. I hope that you find some interesting results within.


\chapter*{Abstract}\label{abstract}

It is the purpose of this thesis to enunciate and prove a collection of explicit results in the theory of prime numbers.

First, the problem of primes in short intervals is considered. We furnish an explicit result on the existence of primes between consecutive cubes. To prove this, we first derive an explicit version of the Riemann--von Mangoldt explicit formula. We then assume the Riemann hypothesis and improve on the known conditional explicit estimates for primes in short intervals. 

Using recent results on primes in arithmetic progressions, we prove two new results in additive number theory. First, we prove that every integer greater than two can be written as the sum of a prime and a square-free number. We then work similarly to prove that every integer greater than ten and not congruent to one modulo four can be written as the sum of the square of a prime and a square-free number. 

Finally, we provide new explicit results on an arcane inequality of Ramanujan. We solve the inequality completely on the assumption of the Riemann hypothesis.

\

  \setlength{\parskip}{0cm}

  \tableofcontents\blanknonumber

  \mainmatter

  

  \setlength{\parskip}{16pt}
  
\chapter{Introduction}
\label{chapter1}

\begin{quote}
\textit{The White Rabbit put on his spectacles. ``Where shall I begin, please your Majesty?'' he
asked.\\
``Begin at the beginning,'' the King said, very gravely, ``and go on till you come to the end: then stop.''} 
\flushright -- Lewis Carroll, \textit{Alice's Adventures in Wonderland}
\end{quote}

It is the author's hope that this introduction will be readable to any person acquainted with a good level of mathematics. As such, one should find that some effort has been put in to banish the technical particulars to later chapters so as not to obscure the basic nature of the work. 

This thesis details the recent period of research undertaken by the author; essentially, we present a collection of new results in the theory of prime numbers. Most of these results have already found their place in the literature through their publication in scholarly journals. Therefore, the primary purpose of this thesis is to present this suite of results cohesively in the context of modern number theory. One cannot do this, of course, until the setting has been drawn, and so the execution of this first chapter is as follows. We will provide an abridged history of the theory of prime numbers before describing the author's contributions.

\section{An abridged history of the prime numbers}

Number theory is an ancient branch of mathematics centred around the positive integers:
$$1,2,3,4,5,6,7,8,9,10,11,12,13,\ldots$$
The allure, perhaps, that this field of study holds for mathematicians and amateurs alike, is that it allows one to formulate and test a hypothesis with apparent ease. That is, one possessing little formal background may still conduct seemingly effective investigations into the properties of numbers. However, much difficulty often ensues when one resolves to prove one's conjecture in a rigorous way. Such a claim can be supported by first introducing the sequence of prime numbers:
$$2,3,5,7,11,13,17,19,23,29,31,37,41,43, \ldots$$
These are the numbers that are greater than one and are not perfectly divisible by any other positive integer except for one and themselves. And so, the prime numbers have now been defined precisely, but if one were to, say, ask for a mathematical formula which identifies the prime numbers in an efficient way, one would be disappointed by an unimpeachable lack of returns. Put simply, the primes seem to avoid the good behaviour that we have come to expect of straightforwardly constructed sequences of numbers. 

One might also suppose it reasonable that the number one should appear in the above list. The short response to this is that it used to be included, but was at some point cleaved from the list in Pluto-esque fashion for the sake of convenience\footnote{It is not entirely clear when this happened, though we note that the number one is listed as a prime in the tables of Lehmer \cite{lehmerlist} and in Northrop's book of paradoxes \cite{northrop}. The article by Caldwell and Xiong \cite{caldwellxiong} discusses this point further in some great detail.}. The typical argument that one puts forward in order to justify such a divorce is to mention the \textit{fundamental theorem of arithmetic}, which states that every positive integer greater than one is either a prime or a \textit{unique} product of primes. For example, one may decompose 60 into a unique product of primes \textit{viz.}
$$60 = 2 \times 2 \times 3 \times 5$$
and this is the only way which one may do so, ignoring swapping the order of multiplication. Now, if the number one were to find itself in the list of primes, then we could write
$$60 = 2 \times 2 \times 3 \times 5 \times 1 \times \cdots \times 1$$
with any count of ones on the tail, and this would disrupt any sort of uniqueness. Therefore, by omitting the number one from the list, the positive integers are uniquely composed of primes through multiplication. 

The next property that can be made of the primes is that there are infinitely many of them. This statement was first proven by Euclid of Alexandria \cite{euclid} around 300BC, and took form as Proposition 20 in the ninth book of Euclid's \textit{Elementa}. 

\begin{theorem}
There are infinitely many prime numbers.
\end{theorem}

\begin{proof}
Assume that there are only finitely many prime numbers that we can list as $p_1, p_2, \ldots, p_k$. If we let $N=p_1 \cdots p_k$ be the product of all the primes, then it is clear that $N+1$ must be divisible by a prime that we have not listed, for any two consecutive integers can have no prime factors in common. This contradicts our assumption that there are finitely many primes.
\end{proof}

It should be noted that this is not Euclid's original proof; actually, Euclid simply delivers a rigorous proof for the case $k=3$, and so kindly introduces the well-adopted practice of leaving the remaining cases as a homework exercise.

It is the author's intention at this early time to give a new proof of Euclid's theorem. This exercise has been done in abundance (see the book of Narkiewicz \cite[Ch.\ 1]{narkiewicz} for a collection of such proofs), the point frequently being some contradiction arising from the assumption of having only a finite set of primes. Here, we reproduce Euler's contradiction that the harmonic series (incorrectly) converges, though we do so through humbler means. 

The reader should know two things in order to follow the proof. First, we say that a number is square-free if it is not divisible by the square of any prime number. As such, one can collect all square-free numbers simply by taking products of distinct primes. One also needs the familiar fact that the harmonic series
$$\sum_{n=1}^{\infty} \frac{1}{n} = 1 + \frac{1}{2} + \frac{1}{3} + \frac{1}{4} + \cdots$$
diverges to infinity. 

\begin{proof}
We begin by assuming that there are only finitely many prime numbers, from which it follows immediately that there are only finitely many square-free numbers. From here, we shall use $A$ to denote the set of square-free numbers.

Thus, it is clear that every sufficiently large integer is divisible by the square of some integer. This means that for any positive integer $N$, the set $\{1,2,3,\ldots,N\}$ can be covered by taking the union of $A$ with all multiples of squares that do not exceed $N$. From this, we obtain an upper bound for the $N$-th harmonic number:
$$\sum_{n \leq N} \frac{1}{n} \leq \sum_{a \in A} \frac{1}{a} + \bigg(\frac{1}{2^2} + \frac{1}{3^2} + \frac{1}{4^2} + \cdots\bigg) \bigg(1 + \frac{1}{2} + \frac{1}{3} + \cdots + \frac{1}{N}\bigg).$$ 
Importantly, term-by-term expansion on the two brackets will recover a series that includes the reciprocals of numbers that are not square-free and do not exceed $N$. The inequality is somewhat weak, in the sense that we will recover many numbers that are greater than $N$. We can bound the sum of the reciprocals of the nontrivial squares as follows:
$$\sum_{m=2}^{\infty} \frac{1}{m^2} < \frac{1}{2^2} + \sum_{m=3}^{\infty} \frac{1}{m(m-1)} = \frac{1}{4} + \sum_{m=3}^{\infty} \bigg( \frac{1}{m-1} - \frac{1}{m}\bigg) = \frac{3}{4}.$$
Therefore, we have upon rearranging that
$$\frac{1}{4} \sum_{n \leq N} \frac{1}{n} < \sum_{a \in A} \frac{1}{a}$$ for all integers $N$. This contradicts the divergence of the harmonic series for some sufficiently large integer $N$ and thus completes the proof.
\end{proof}

In 1737, Euler reproved Euclid's theorem through the remarkable identity
\begin{equation} \label{eulerproduct}
\sum_{n=1}^{\infty} \frac{1}{n^{\sigma}} = \prod_{p} (1-p^{-\sigma})^{-1},
\end{equation}
where $\sigma$ is a real number satisfying $\sigma> 1$ and the product is over all prime numbers. To verify this informally, one can subject each member of the product to the Taylor expansion
$$\frac{1}{1-x} = 1 + x + x^2 + x^3 + \cdots$$
which holds for $|x| < 1$, and then expand the result; equality will immediately follow from the fundamental theorem of arithmetic. This argument is not airtight (see Titchmarsh \cite[Ch.\ 1]{titchmarsh1986theory} for a rigorous proof), but it will suffice to impart the main idea, namely that (\ref{eulerproduct}) can be thought of as an analytic version of the fundamental theorem of arithmetic, for we have encapsulated this property of the natural numbers into an identity which holds for all $\sigma >1$. The infinitude of primes follows from taking the limit as $\sigma \rightarrow 1^+$; the left hand side of (\ref{eulerproduct}) tends to infinity, and so we deduce that the product on the right must also be infinite, and thus composed of infinitely many parts.

In fact, Euler went somewhat further than this. If we let $\sigma>1$, then taking the logarithm of both sides in (\ref{eulerproduct}) gives us that
\begin{eqnarray*}
\log \bigg(\sum_n \frac{1}{n^{\sigma}}\bigg) & = &  - \sum_p \log (1-p^{-\sigma}). 
\end{eqnarray*}
From an application of the Taylor series
$$-\log(1-x) = \sum_{n=1}^{\infty} \frac{x^n}{n}$$
which holds for all $|x| < 1$, it follows that
\begin{eqnarray*}
\log \bigg(\sum_n \frac{1}{n^{\sigma}}\bigg) & = &  \sum_p \bigg( \frac{1}{p^{\sigma}} + \frac{1}{2p^{2{\sigma}}} + \frac{1}{3p^{3{\sigma}}}+ \cdots \bigg) \\
& = & \sum_p  \frac{1}{p^{\sigma}}  +  \sum_p \frac{1}{p^{2{\sigma}}} \bigg(\frac{1}{2} + \frac{1}{3p^{\sigma}} + \frac{1}{4p^{2{\sigma}}}+\cdots \bigg) \\
& < & \sum_p  \frac{1}{p^{\sigma}}  +  \sum_p \frac{1}{p^{2{\sigma}}} \bigg(1 + \frac{1}{p^{{\sigma}}} +\frac{1}{p^{2{\sigma}}} + \cdots \bigg).
\end{eqnarray*}
It is straightforward to prove that the rightmost sum is bounded above by some large positive number $C$. It follows that
\begin{eqnarray*}
 \sum_p  \frac{1}{p^{\sigma}}  >  \log \bigg(\sum_n \frac{1}{n^{\sigma}}\bigg) - C, 
\end{eqnarray*}
and so as $\sigma \rightarrow 1$, the right hand side tends to infinity, impressing the sum over the primes towards infinity as well. This is a stronger result than the infinitude of primes.

It is convenient at this point to introduce a modern notation. We will let $\pi(x)$ denote the number of primes that do not exceed $x$, and we will henceforth refer to this as the \textit{prime counting function}. This allows us to recognise Euclid's theorem in the form
$$\lim_{x \rightarrow \infty} \pi(x) = \infty.$$
It is a central problem in number theory to resolve the function $\pi(x)$ to a greater detail. In 1808, Legendre \cite{legendre} asserted that the function $\pi(x)$ is approximately equal in value to
$$\frac{x}{\log x - A}$$
for large values of $x$ and $A = 1.08$. Gauss proposed a similiar though different formula, specifically that $\pi(x)$ should be approximated by the so-called offset logarithmic integral
$$\text{Li}(x) : = \int_2^x \frac{du}{\log u}.$$
Such an observation was communicated from Gauss to Encke in 1849, but Gauss' work towards this statement seemed to have commenced as early as 1791, when he was a boy at the age of fourteen. In fact, it is said that around this early time, the Duke of Brunswick gave Gauss a collection of mathematics books; these included tabulations of logarithmic values. Gauss worked to improve these tables, in fact, a computation of $\log(10037)$ can be found in his papers from the era. It has been said (see, for example, Tschinkel \cite{tschinkel}) that his fervent work on these tables probably led him to the statement that $\pi(x)$ and $\text{Li}(x)$ are approximately equal in value.

One can expand Legendre's estimate via a Taylor series to get
$$\frac{x}{\log x - A} = \frac{x}{\log x} + \frac{A x}{\log^2 x} + \frac{A^2 x}{\log^3 x} + \frac{A^3 x}{\log^4 x} + \cdots$$
whereas integration by parts gives
$$\int_2^x \frac{du}{\log u} = \frac{x}{\log x} + \frac{x}{\log^2 x} + \frac{2 x}{\log^3 x} + \frac{6 x}{\log^4 x} + \cdots$$
for the estimate of Gauss. Therefore, both estimates are of the form
$$\frac{x}{\log x} + O \bigg(\frac{x}{\log^2 x}\bigg)$$
and so it at least seemed plausible that the function $x/ \log x$ was a good first approximation for the function $\pi(x)$. Chebyshev \cite{chebyshev} made the first serious move towards a proof of this. In 1848, he proved the result that if the limit
\begin{equation} \label{chebyshevintegral}
\lim_{x \rightarrow \infty} \frac{\pi(x)}{x/\log x}
\end{equation}
exists then it must be equal to one. Famously, he was also able to prove the inequality
$$\frac{a}{\log x} < \pi(x) < \frac{Ax}{\log x}$$
for certain positive numbers $a$ and $A$ and all sufficiently large values of $x$. Specifically, he showed that one could take $a=0.921$ and $A=1.106$. Historically, there have been some attempts to refine these values using variations of Chebyshev's method, though it seems clear that this line of attack falls short of showing that the limit in (\ref{chebyshevintegral}) tends to one.

It should be said that Chebyshev's work allowed him to prove a theorem that is well-known as Bertrand's postulate. This states that for any $x > 1$ there exists a prime number $p$ that satisfies $x < p < 2x$. We will return to this result soon, for it is a good starting point with respect to some of the work of this thesis.

In 1859, Riemann published his only paper \cite{riemann} in number theory, enunciating a new way to investigate the prime numbers. He defined, for all complex numbers $s$ in the half-plane $\text{Re}(s)>1$, the complex-valued function 
$$\zeta(s) =\sum_{n=1}^{\infty} \frac{1}{n^s},$$
which we hereby refer to as the Riemann zeta-function. It should be noted that we employed this function earlier to prove that the series of reciprocal primes diverges, though we considered only the case where $s$ was real. 

Now treating $s$ as a complex variable, Riemann indicated how $\zeta(s)$ could be extended to a regular (single-valued and analytic) function everywhere except for a simple pole at $s=1$. Moreover, he stated that the zeroes of $\zeta(s)$ are inextricably connected to the distribution of the prime numbers, and sketched a proof that 
\begin{equation} \label{PNT}
\lim_{x \rightarrow \infty} \frac{\pi(x)}{x/ \log x}=1
\end{equation}
amongst other things. These blueprints were finally completed close to thirty years later by Hadamard \cite{hadamard} and by de la Vall\'{e}e Poussin \cite{poussin}, who finished their work independently at around the same time. It is to these two mathematicians, building on the foundation laid by Riemann and the others before him, that we attribute the proof of the Prime Number Theorem, that is, the confirmation of the limit in (\ref{PNT}).

As mentioned earlier, such a result is elicited from the behaviour of the zeroes of the Riemann zeta-function. Now, it can be shown with some work (see Titchmarsh \cite[Ch.\ 2]{titchmarsh1986theory}) that any zero $\rho$ of $\zeta(s)$ will either be a negative even integer or will be constrained to lie in the so-called critical-strip 
$$0\leq\text{Re}(\rho)\leq 1.$$ 
It is precisely these nontrivial zeroes which lie in the critical strip that form our focus, for they weave into the zeta-function all information regarding the distribution of the primes. 

Hadamard's proof of the Prime Number Theorem follows from showing that $0 < \text{Re}(\rho) < 1$, that is, by whittling down the width of the critical strip albeit infinitesimally. However, de la Vall\'{e}e Poussin was more precise with his workings. Specifically, he proved that there exists $c>0$ and $K >1$ such that
$$\frac{c}{\log \gamma} \leq \beta \leq 1 - \frac{c}{\log \gamma}$$
for all $\gamma > K$. It can, in fact, be deduced from this that
$$| \pi(x) - \text{Li}(x)| < \text{Li}(x) e^{- \sqrt{A \log x}}$$
for some $A>0$ and all sufficiently large $x$. This proves that the logarithmic integral proposed by Gauss is the correct model to hold against the prime counting function. 

At this point, we have given a brief overview of the development of the Prime Number Theorem, and we have hinted at some of the interplay between zeroes and primes. In the last century, number theory has burgeoned incredibly, and it would be an insurmountable task to provide a decent survey in this section. This is, however, done in an impressively thorough fashion by S\'{a}ndor, Mitrinovi\'{c} and Crstici \cite{sandor}. We will, instead, discuss the branches of the theory that have been the subject of the author's investigations. Before we do this, however, it is necessary to provide some material on explicit results.

\section{Motivation for explicit results}

It is clear that many results in number theory, such as de la Vall\'{e}e Poussin's proof of the Prime Number Theorem, are stated with the use of implied constants, where one encounters qualifiers such as ``sufficiently large''. There are, in fact, three types of results that we see in number theory:
\begin{enumerate}
\item An \textit{ineffective} result shows that a statement is true for an unspecified constant $C$, but one cannot actually determine the constant $C$ by reworking the proof and keeping track of the error terms. 
\item An \textit{effective but not explicit} result shows that a statement is true for an unspecified constant $C$ with the bonus that one could actually determine a suitable value for $C$ by reworking the proof explicitly. 
\item An \textit{explicit} result gives a numerical value for $C$. 
\end{enumerate}

This thesis places regard on establishing explicit results in the theory of prime numbers. It is said that in this area, one often works through the original proof of some result whilst keeping careful bounds on any error terms that arise. However, in some cases the original proof is ineffective and so one must come up with a new but effective proof first. In fact, one of the results enunciated in this thesis requires such a remodelling; this will be discussed in more detail in Chapter 4.

It should be clear that the type of some result will depend directly on the methods used to prove it. For example, consider Bertrand's postulate, the statement that there is a prime in the interval $(x,2x)$ for all $x>1$. This result is explicit and, as such, would require explicit results to prove it, namely some reasonable estimate on the prime counting function. 

To elaborate, we can consider the Prime Number Theorem in the ineffective form as proven by Hadamard, which states that the quotient of $\pi(x)$ over $x /\log x$ tends to one as $x$ tends to infinity. It follows \textit{a fortiori} that for every $\epsilon > 0$ there exists some $x_0(\epsilon) > 0$ such that
$$\bigg| \pi(x) - \frac{x}{\log x} \bigg| <  \frac{\epsilon x}{\log x}$$
for all $x > x_0(\epsilon)$. Thus, choosing $\epsilon =1/4$ and letting $x > x_0(1/4)$ we have that
\begin{eqnarray*}
\pi(2x) - \pi(x) & > & \frac{2x}{\log (2x)} - \frac{x}{\log x} - \frac{x}{2 \log (2x)} - \frac{x}{4 \log x} \\
& > & \frac{x}{5 \log x}
\end{eqnarray*}
provided also that $x > x_1$, where $x_1$ is sufficiently large so that the second inequality holds. Therefore, we have that there is a prime in the interval $(x,2x)$ provided that $x > \max(x_0(1/4), x_1)$. This form of Bertrand's postulate is ineffective, and this is simply because the ineffectiveness of Hadamard's result is pushed through into our proof. And whilst it is a straightforward problem to determine a suitable value for $x_1$, Hadamard's proof of the Prime Number Theorem does not deliver the details on how $x_0(\epsilon)$ behaves as a function of $\epsilon$.

It is at this point that we provide some comfort to the reader; the Prime Number Theorem has been realised explicitly since the time of Hadamard and de la Vall\'{e}e Poussin. A more detailed discussion of this will be found in Chapter 2. For now, it remains to motivate the reason for developing explicit results.

In a pure sense, some statements in number theory may seem more complete when stated explicitly. If one can prove, for example, that every sufficiently large integer is endowed with some property of interest, then surely the next step in strengthening such a result is to remove the qualification of being sufficiently large. An immediate barrier to this could be that the proof itself is ineffective, or that the proof is too difficult to be worked through carefully. One must always weigh the labour of furnishing an explicit result against the value of possessing it; personal taste is clearly a deciding factor here.

However, we now find ourselves in an age where such results are not only easier to obtain but also more valuable than before. The area of explicit methods in number theory is gathering momentum in the wake of high-speed computation and an enhanced interest in the prime numbers for application. To flesh out this claim with an example, it is known that certain properties of the prime numbers allow us to construct very efficient computer networks (see the survey on expander graphs \cite{expandersurvey} by Hoory, Linial and Wigderson). However, if one only knows that these properties hold for ``sufficiently large'' prime numbers, then one can only guarantee high connectivity in networks with a ``sufficiently large'' number of computers. This is a barrier to having any measure of practical assurance. An explicit result, on the other hand, will allow us to state exactly which primes hold the property of interest, and this in turn allows us to build good networks in practice. 

Other areas of mathematics will, at certain times, also draw on the properties of the integers, and thus the development of explicit results in number theory finds utility in research of a purely mathematical nature. For instance, a group theorist seeking explicit bounds on the maximum order of an element in the symmetric group on $n$ letters can achieve this via means of explicit estimates on the prime numbers. One can see the work of Massias, Nicolas and Robin \cite{massiasNicolasRobin} for the details of this profound relationship. 

\section{Summary of thesis}

This thesis describes a collection of new explicit results in the theory of prime numbers.  It follows from the work of Hadamard and de la Vall\'ee Poussin that if one wants to furnish such results on the primes, one first needs to establish explicit bounds on the zeroes of the Riemann zeta-function. The general framework for this approach is described in detail in Chapter 2, but we will outline the results here.

To connect the prime numbers to the zeroes of $\zeta(s)$, we require the use of a so-called explicit formula. It is important to note that here we do not mean explicit in the sense of determining explicit constants; the formula explicitly connects some sum over the prime numbers to another sum over the zeroes of $\zeta(s)$. There are various types of explicit formulas that one can use, however we will be considering principally the Riemann--von Mangoldt explicit formula, which states that
$$\psi(x)=\sum_{p^m \leq x} \log p = x - \sum_{\rho} \frac{x^{\rho}}{\rho} - \log 2 \pi - \frac{1}{2} \log(1-x^{-2}),$$
where the leftmost sum is over the prime powers that do not exceed some positive non-integer $x$, and the other sum is over the nontrivial zeroes $\rho = \beta+i \gamma$ of $\zeta(s)$. Clearly, knowledge regarding the distribution of zeroes can be injected into the rightmost sum in order to gain clarity regarding the prime numbers. More useful, however, is the truncated Riemann--von Mangoldt explicit formula
$$\psi(x)= x - \sum_{| \gamma | \leq T} \frac{x^{\rho}}{\rho} + O\bigg( \frac{x \log^2 x}{T} \bigg),$$
which allows its employer to insert information on the nontrivial zeroes with height at most $T$. In Chapter 2, we prove an \textit{actually} explicit version of this formula; this allows one to keep numerical bounds on the error term. We do this by working through the classical version of the proof whilst working carefully to control the error terms.

In Chapter 3, we show how our explicit formula can be used to furnish a result on primes in short intervals. Notably, at the International Congress of Mathematicians in 1912, Landau listed four basic problems about the prime numbers. The third of these was a conjecture of Legendre's that there exists a prime between any two consecutive squares. This is currently unresolved, however it was shown by Ingham \cite{ingham} in 1937 that there exists a prime between any two sufficiently large cubes. We make Ingham's result explicit; specifically, we show that there will be a prime in the interval $(n^3, (n+1)^3)$ provided that $n \geq \exp(\exp(33.3))$.

This result is achieved by constructing an explicit formula that detects primes in intervals, and then substituting into it explicit information regarding the real parts of the nontrivial zeroes of $\zeta(s)$. It should be stressed at this point, that most successes in bounding the real parts of these zeroes are weak in comparison to the famed hypothesis of Riemann, which asserts that every zero in the critical strip has a real part of~$1/2$. In fact, working under the Riemann hypothesis, it can be shown that there will be a prime between \textit{any} two consecutive positive cubes (see the work of Caldwell and Cheng \cite{caldwellcheng}). Our result, therefore, is a step towards proving the theorem of Caldwell and Cheng without the conditional sledgehammer of Riemann's hypothesis. 

On a related note, it is almost conventional for number theorists to examine a result unconditionally first, and then to see what can be done further on the assumption of the Riemann hypothesis. The author has found this convention to be a tempting one, and so some of the problems considered in this thesis are doubly assaulted. The problem of primes in short intervals is no exception to this, and so we take such an approach in the second part of Chapter 3.

It is already a celebrated result of Cram\'{e}r \cite{cramer} from 1921 that, assuming the Riemann hypothesis, there exist some $c>0$ and $x_0$ such that the interval $(x,x+c \sqrt{x} \log x]$ contains a prime number for all $x>x_0$. Using explicit formulas, we prove that there will be a prime in the interval $(x-\frac{4}{\pi} \sqrt{x} \log x,x]$ for all $x\geq 2$. Moreover, we show that the constant $4/\pi$ can be reduced to $1+\epsilon$, where $\epsilon>0$ can be taken to be arbitrarily close to zero. This result has been published by the author in the \textit{International Journal of Number Theory} \cite{dudekcramer}. Moreover, the author has refined this result in a paper coauthored with L. Greni\'{e} and G. Molteni \cite{dudekmoltenigrenie}. This paper is due to appear in the same journal. 

In Chapter 4, the author shifts his attention towards additive problems involving primes. Such problems are borne from the zeal of number theorists in supposing that, as every number may be composed multiplicatively by the primes, it might also be true that one can build numbers from primes using addition. The root of this can be traced to a letter of Goldbach's dating back to 1742, in which he writes to Euler and conjectures that every even number greater than 2 can be written as the sum of two primes, and every odd number greater than 5 can be written as the sum of three primes. The latter conjecture was proven by Helfgott \cite{helfgott2013} in 2013, and thus it is only the so-called binary Goldbach conjecture which remains. This problem also made it onto the famous problem list of Landau's which we referred to earlier.

The best known result towards the binary Goldbach conjecture is a theorem of Chen \cite{chen} which states that every sufficiently large even integer can be written as the sum of a prime and a product of at most two primes. It was the author's original intention to make Chen's theorem explicit, perhaps by proving that all even numbers greater than two could be written this way. A result towards this was recently established by Yamada \cite{yamada}, where he proves that every even integer greater than $\exp(\exp(36))$ can be written this way. An earlier but weaker result of Estermann \cite{estermann} seemed to possess the appropriate return-on-effort. As such, the author proves in the first part of Chapter 4 that every integer greater than two is the sum of a prime and a square-free number. This result has been accepted for publication and will appear in the \textit{Ramanujan Journal} \cite{dudekestermann}.

In the second part of Chapter 4, we consider an extension of Estermann's result which was first studied by Erd\H{o}s \cite{erdos}. Specifically, Erd\H{o}s shows that any sufficiently large positive integer $n \not\equiv 1 \text{ mod } 4$ can be written as the sum of the square of a prime and a square-free number. In the same nature as our continuation of Estermann's work, we prove that every integer $n\geq 10$ such that $n \not\equiv 1 \text{ mod } 4$ can be written as the sum of the square of a prime and a square-free number. This was collaborative work with Dave Platt at the University of Bristol and has been accepted for publication by the \textit{LMS Journal of Computation and Mathematics}~\cite{dudekplatt2}. 

Chapter 5 details another collaboration with Dave Platt, where we study a curious inequality first penned by the prodigious mathematician Ramanujan. In one of his famous notebooks (see Berndt \cite{berndt} for the details), Ramanujan writes that
$$\pi^2(x) <\frac{ex}{\log x} \pi \bigg(\frac{x}{e} \bigg)$$
holds for all sufficiently large values of $x$. We take it upon ourselves to make this explicit, proving that the inequality holds unconditionally for all $x \geq \exp(9394)$. We then place the inequality under the conditional scrutiny of the Riemann hypothesis, from which it follows that $x=38, 358, 837, 682$ is the largest integer counterexample to Ramanujan's inequality. This solves the inequality completely under the assumption of the Riemann hypothesis. This work has been published in \textit{Experimental Mathematics} \cite{dudekplattem}.

Finally, we use Chapter 6 to outline some projects which are similar to those in this thesis and could be undertaken by eager researchers. These are projects which once found themselves on the author's interminable to-do list, but would probably best serve the field by tempting in young researchers or being the talking point for new collaborations.


\chapter[The Explicit Formula]{The Riemann--von Mangoldt Explicit Formula}
\label{chapter2}

\begin{quote}
\textit{``Take some more tea," the March Hare said to Alice, very earnestly.} \\
\textit{``I've had nothing yet,'' Alice replied in an offended tone, ``so I can't take more.''}\\
\textit{``You mean you can't take less,'' said the Hatter: ``it's very easy to take more than nothing.''}

\end{quote}

\section{The Riemann zeta-function}

This chapter studies the relationship between the prime numbers and the zeroes of the Riemann zeta-function. In this first section, the focus is emphatically on the latter, and we deliver a terse course on the properties of the zeta-function, referring readers to the excellent text of Titchmarsh \cite{titchmarsh1986theory} for more details. 

Throughout this chapter, we will often refer to the complex variable $s = \sigma + i t$. As mentioned in the introduction, the Riemann zeta-function is defined for complex numbers satisfying $\text{Re}(s) > 1$ by the infinite series
$$\zeta(s) = \sum_{n=1}^{\infty} \frac{1}{n^s}.$$
One can continue this analytically to the entire complex plane with the exception of a simple pole at the point $s=1$. Moreover, it can be shown that the so-called functional equation
$$\pi^{-s/2} \Gamma\bigg( \frac{s}{2} \bigg) \zeta(s) = \pi^{-\frac{1-s}{2}} \Gamma\bigg( \frac{1-s}{2} \bigg) \zeta(1-s)$$
holds, where $\Gamma$ denotes the gamma function. Noting that both $\zeta(s)$ and $\Gamma(s)$ are without poles in the half-plane $\text{Re}(s)>1$, it follows that the simple poles of $\Gamma(s/2)$ at the negative even integers give rise to zeroes of $\zeta(s)$ at the same points. 

Moreover, it is known that all other zeroes $\rho = \beta+i \gamma$ of $\zeta(s)$ are constrained to lie in the aforementioned critical strip, that is, they satisfy the bound $0 < \beta < 1$. It is also known (from the functional equation) that these zeroes are distributed symmetrically about the line $\text{Re}(s) = 1/2$. The holy grail of analytic number theory would be the Riemann hypothesis, which asserts that $\beta=1/2$ for all zeroes in the strip. In lieu of such a grand theorem, there is an extensive line of work towards establishing measures on the distribution of the nontrivial zeroes.

One way to quantify this distribution is through the use of a \textit{zero-free region} in the critical strip. For example, it was first shown by de la Vall\'ee Poussin \cite{poussin}, in his proof of the Prime Number Theorem, that there exists $c>0$ such that every zero in the critical strip satisfies
$$\frac{c}{\log | \gamma |} \leq \beta \leq 1-  \frac{c}{\log | \gamma |}.$$
The best known zero-free region was given by Korobov \cite{korobov} and Vinogradov \cite{vinogradovzero}. Specifically, they show that there exists $c>0$ such that any zero in the strip satisfies
\begin{equation} \label{vinogradovkorobov}
\frac{c}{(\log | \gamma |)^{2/3} ( \log \log |\gamma|)^{1/3}} \leq \beta \leq 1-  \frac{c}{(\log | \gamma |)^{2/3} ( \log \log |\gamma|)^{1/3}}.
\end{equation}
We will need to make use of several explicit versions of these zero-free regions at certain places within this thesis. 

Another way to probe the distribution of zeroes is through the use of a \textit{zero-density estimate}. Specifically, we let $N(T)$ denote the number of zeroes in the critical strip such that $0 < \gamma < T$. Then we denote by $N(\sigma,T)$ the number of these zeroes with $\beta >\sigma$. A zero-density estimate is then some upper bound on the function $N(\sigma,T)$. A simple example can be found in Theorem 9.15(A) of Titchmarsh \cite{titchmarsh1986theory}, where it is shown that for any fixed $\sigma > 1/2$ we have that
$$N(\sigma,T) = O(T).$$
In comparison to the well known asymptotic formula
$$N(T) = \frac{T \log T}{2\pi}+O(T),$$
this shows us that all but an infinitesimal proportion of zeroes of $\zeta(s)$ satisfy $\frac{1}{2} - \epsilon < \beta < \frac{1}{2} + \epsilon$ for any $\epsilon > 0$.

In Chapter 3, we will use explicit versions of both zero-free regions and zero-density estimates to attack the problem of primes in short intervals. First, we need to prove the formula that permits such analytic results to push through into the prime numbers.

\section{Proof of the explicit formula}
In this section, we derive the Riemann--von Mangoldt explicit formula with an explicit bound on the error term. Our considerations begin with the von Mangoldt function:
\begin{displaymath}
   \Lambda(n) = \left\{
     \begin{array}{ll}
       \log p  & : \hspace{0.1in} n=p^m, \text{ $p$ is prime, $m \in \mathbb{N}$}\\
       0   & : \hspace{0.1in} \text{otherwise.}
     \end{array}
   \right.
\end{displaymath} 
As mentioned in Chapter 1, we introduce the sum $\psi(x) = \sum_{n \leq x} \Lambda(n)$ to study the distribution of the prime numbers; this is known as the Chebyshev $\psi$-function. The Riemann--von Mangoldt explicit formula can be stated thus
\begin{equation} \label{RVM}
\psi(x) = x - \sum_{\rho} \frac{x^\rho}{\rho}-\log 2\pi - \frac{1}{2} \log(1-x^{-2}),
\end{equation}
where $x$ is any positive non-integer and the sum is over all nontrivial zeroes $\rho$ of $\zeta(s)$ (see Davenport \cite[Ch.\ 17]{davenport} for details). One can see that feeding information regarding the zeroes of $\zeta(s)$ into the above formula yields estimates for the prime powers, allowing us then to obtain estimates regarding the primes. However, the explicit formula in (\ref{RVM}) relies on estimates over \textit{all} of the nontrivial zeroes of $\zeta(s)$ and so is impractical for certain applications.

We often find more use in a \textit{truncated} version of the explicit formula in which one inputs information regarding the zeroes $\rho$ up to some height $T$, that is, with $| \text{Im}(\rho) | < T$. It is our first intention to provide such a formula but with an explicit error term; this will allow us to make progress on the problem of primes between consecutive cubes. To state such a result, we will first define the notation $f(x) = g(x) + O^*(h(x))$ to mean $|f(x)-g(x)| < h(x)$ for some given range of $x$. Our result is as follows.

\begin{theorem} \label{explicitformula}
Let $x>e^{60}$ be half an odd integer and suppose that $50 < T < x$. Then
\begin{equation} \label{explicitequation}
\psi(x) = x - \sum_{|\text{Im}( \rho)| < T} \frac{x^{\rho}}{\rho} + O^*\bigg( \frac{ 2 x \log^2 x}{T}\bigg).
\end{equation}
\end{theorem}

In the remainder of this chapter, unless otherwise mentioned, we shall assume that $x$ and $T$ are as specified in Theorem \ref{explicitformula}. Notably, the requirement that $x$ is half an odd integer minimises the error term in Theorem \ref{explicitformula}, as will become evident in the proof of Lemma $\ref{boundbigsum}$. The reason for the seemingly arbitrary bound of $x>e^{60}$ is as follows: we will want to prove that there is a prime in the interval $(x,x+3 x^{2/3})$, as this interval is related to the problem of primes between cubes. It turns out that we can do this if $x \leq e^{60}$ using the short-interval results provided by Ramar\'{e} and Saouter \cite{ramaresaouter}. 

In some of our theorems we drop these requirements. This is because we thought the result would be useful to other mathematicians if stated with fewer constraints, or that the error terms were reasonably small without needing $x$ to be quite so large. For example, Kerr \cite{kerr} has used Lemma \ref{bryce} to establish new lower bounds for the Riemann zeta-function on short intervals of the line $\text{Re}(s) = 1/2$. At other times, there is a clear way to improve a result that has not been pursued. In many of these cases, the work involved more than offsets the almost-negligible rewards, and so we eschew the messy in favour of the neat. We will, however, provide details in such cases for the interested reader.

It should be noted that before working on this problem, the author attempted to locate an effective explicit formula in the literature and found that which was given by Skewes \cite{skewes}, though the error term here is of better order. Liu and Wang \cite{liuwang} give a version of Theorem \ref{explicitformula} with an improved constant, but holding only for $T$ as a certain function of $x$, which is useful for explicit estimates on the ternary Goldbach conjecture but not for our application.

The method of proof for the explicit formula is well-known: we employ Perron's formula to express $\psi(x)$ as a contour integral over a vertical line. We then truncate this integral to one that is over a finite line segment. This is where we will pick up the bulk of the error term, and so more precision is needed here than anywhere else in our proof. Our line of integration is then shifted so as to acknowledge the residues and introduce the sum over the zeroes which accompanies $\psi(x)$ in Theorem \ref{explicitformula}. The crux of our working involves keeping track of the errors.

We proceed as laid out in Davenport's \cite{davenport} well known expository text, though by working carefully we will obtain an explicit result.  We should also note that this derivation can be applied to other arithmetic functions, though there are some differences to be noted. In our case, we have the identity
$$- \frac{\zeta'(s)}{\zeta(s)} = \sum_{n=1}^{\infty} \frac{\Lambda(n)}{n^s}$$
for $\text{Re}(s)>1$, which ties the von Mangoldt function to the Riemann zeta-function. It is precisely this relationship that permits techniques from complex analysis to probe the behaviour of $\psi(x)$. In this case, properties and bounds pertaining to $\zeta(s)$ have ramifications for the prime numbers. Note that in a more general context, one has an arithmetic function $a:\mathbb{N} \rightarrow \mathbb{C}$ and a \textit{Dirichlet series}
$$F(s) = \sum_{n=1}^{\infty} \frac{a(n)}{n^s},$$
convergent for $\sigma > \sigma_0$. We will now show where the complex analysis comes in. For $c>0$, we define the contour integral
$$\delta(x) = \frac{1}{2 \pi i} \int_{c-i\infty}^{c+i\infty} \frac{x^s}{s} ds.$$
A good exercise (see Murty's \cite{murty} problem book) for budding students of analytic number theory (or complex analysis) is to show that
\begin{equation*}
\delta(x) =
\left\{
	\begin{array}{ll}
		0  & \mbox{if } 0 < x < 1 \\  
		1/2 & \mbox{if } x=1 \\  
		1 & \mbox{if } x > 1.
	\end{array}
\right.
\end{equation*}
The importance of this integral becomes apparent when one wishes to study the sum of an arithmetic function up to some value $x$, particularly when that function is generated by a Dirichlet series. In our case, we consider the following for a positive non-integer $x$ and $c>1$:
\begin{eqnarray} \label{boogers}
\psi(x) = \sum_{n \leq x} \Lambda(n) & = &  \sum_{n = 1}^{\infty} \Lambda(n) \delta\Big(\frac{x}{n} \Big) \nonumber \\
& = & \sum_{n=1}^{\infty} \Lambda(n) \bigg[ \frac{1}{2 \pi i} \int_{c-i \infty}^{c+i \infty} \bigg( \frac{x}{n} \bigg)^s \frac{ds}{s} \bigg] \nonumber \\
& = & \frac{1}{2 \pi i} \int_{c-i \infty}^{c+i \infty} \bigg( \sum_{n=1}^{\infty} \frac{\Lambda(n)}{n^s} \bigg) \frac{x^s}{s} ds.
\end{eqnarray}
Notice that keeping $c>1$ gives absolute convergence to the series in the above equation, and thus justifies the interchange of integration and summation. From before, we know that the Dirichlet series in (\ref{boogers}) is equal to $-\zeta'(s)/\zeta(s)$, and so we have that
$$\sum_{n \leq x} \Lambda(n) = \frac{1}{2 \pi i}  \int_{c-i \infty}^{c+i \infty} \bigg( - \frac{\zeta'(s)}{\zeta(s)} \bigg) \frac{x^s}{s} ds.$$
In its more general form pertaining to an unspecified arithmetic function this is known as Perron's formula. We may thus estimate the sum of the von Mangoldt function through some knowledge of certain analytic properties of $\zeta'(s)/\zeta(s)$. Our first step is to truncate the path of the integral to a finite segment, namely $(c-iT, c+iT)$. We define for $T > 0$ the truncated integral
$$I(x,T) =  \frac{1}{2 \pi i} \int_{c-i T}^{c+i T} \frac{x^s}{s} ds.$$

The next lemma is a variant of the first lemma in Davenport \cite[Ch.\ 17]{davenport}, and will bound the induced error term upon estimating $\delta(x)$ by $I(x,T)$. The proof is omitted here, though one can see Theorem 15 of Estermann \cite{estermannbook} for a complete proof.

\begin{lemma}
For $x>0$ with $x \neq 1$, $c>0, T>0$ we have
\begin{equation*}
\delta(x) = I(x,T) + O^*\bigg( \frac{x^c}{\pi T | \log x |}  \bigg).
\end{equation*}

\end{lemma}

We now employ this bound to estimate $\psi(x)$ in the following way.
\begin{eqnarray*}
\psi(x) & = &  \sum_{n = 1}^{\infty} \Lambda(n) \delta\bigg(\frac{x}{n} \bigg) \\ 
& = &  \sum_{n = 1}^{\infty} \Lambda(n) \bigg[I\bigg(\frac{x}{n},T\bigg) + O^*\bigg(\frac{1}{\pi T} \bigg(\frac{x}{n}\bigg)^c  \bigg| \log \frac{x}{n} \bigg|^{-1} \bigg)\bigg] \\ 
& = &  \frac{1}{2 \pi i} \int_{c-i T}^{c+i T} - \frac{\zeta'(s)}{\zeta(s)} \frac{x^s}{s} ds + \frac{1}{\pi T} O^* \bigg( \sum_{n=1}^{\infty}  \Lambda(n) \bigg(\frac{x}{n}\bigg)^c \bigg| \log \frac{x}{n} \bigg|^{-1} \bigg).
\end{eqnarray*}
We proceed to bound the sum in the above formula, by splitting it up and estimating it in parts. 

\begin{lemma} \label{boundbigsum}
Let $x > e^{60}$ be half an odd integer and set $c = 1+1/\log x$. Then
\begin{equation}
\sum_{n=1}^{\infty} \Lambda(n) \bigg( \frac{x}{n} \bigg)^c \bigg| \log \frac{x}{n} \bigg|^{-1} < 3.1 x \log^2 x.
\end{equation}

\end{lemma}

\begin{proof}
Some care needs to be taken here. When $x$ and $n$ are quite close, the $\log$ term in the sum will become large. Thus, we introduce the parameter $\alpha \in (1,2)$ and break up the infinite sum:
\begin{eqnarray*}
\sum_{n=1}^{\infty} & = & \sum_{n=1}^{[x/\alpha]} + \sum_{n=[x/\alpha] + 1}^{[x]-1} + \sum_{n = [x]}^{[x]+1} + \sum_{n = [x]+2}^{[\alpha x]}+\sum_{n=[\alpha x]+1}^{\infty} 
\end{eqnarray*}
On the right side of the above formula, denote the $i$th sum by $S_i$. The reader should be convinced by this division; $S_3$ deals with the most inflated terms, namely when $n$ is either side of $x$. Then $S_2$ and $S_4$ measure the remainder of the region which is \textit{close} to $x$. We also note that $S_1$ and $S_5$ contribute little and can be estimated almost trivially.

Considering the range of $n$ in $S_1$ and $S_5$, we have
$$\bigg| \log \frac{x}{n} \bigg| > \log \alpha.$$
Inserting this into these sums, pulling out terms which are independent of $n$, and extending the range of summation to $\mathbb{N}$ we arrive at
\begin{eqnarray} \label{nooborz}
S_1 + S_5 < \frac{x^c}{\log \alpha} \sum_{n=1}^{\infty}\frac{\Lambda(n)}{n^c} & = & \frac{x^c}{\log \alpha} \bigg( -\frac{\zeta'(c)}{\zeta(c)} \bigg) \nonumber \\
& = & \frac{ex}{\log x} \bigg( - \frac{\zeta'}{\zeta} (1+1/ \log x ) \bigg).
\end{eqnarray}
The main theorem from Delange \cite{delange} states that
\begin{equation} \label{fordlemma}
\bigg| - \frac{\zeta'}{\zeta} (\sigma + i t) \bigg| < \frac{1}{\sigma-1}
\end{equation}
for all $\sigma \in (1,1.06]$ and $t \in \mathbb{R}$. We apply this to (\ref{nooborz}) to get
\begin{equation} \label{s1s5}
S_1+S_5 < \frac{ex \log x}{\log \alpha}.
\end{equation} 

We now turn our attention to $S_3$, which is the sum of only two things. It follows, using the fact that $[x]=x-1/2$ and the trivial bound $\Lambda(n) \leq \log n$, that
\begin{eqnarray*}
S_3 & = & \bigg( \frac{x}{x-\frac{1}{2}} \bigg)^c \Lambda(x-1/2) \bigg| \log \frac{x}{x-\frac{1}{2}} \bigg|^{-1} +  \bigg( \frac{x}{x+\frac{1}{2}} \bigg)^c \Lambda(x+1/2) \bigg| \log \frac{x}{x+\frac{1}{2}} \bigg|^{-1} \\
& < & 2 \bigg( \frac{x}{x-\frac{1}{2}} \bigg)^c \log(x+1/2) \Big( \log x - \log (x-1/2) \Big)^{-1} 
\end{eqnarray*}
We can estimate trivially with
$$\log x- \log (x-1/2) = \int_{x-1/2}^{x} \frac{dt}{t} > \frac{1}{2(x-1/2)}$$
and the bound
$$\bigg( \frac{x}{x-\frac{1}{2}} \bigg)^c < 1.1$$
for $x>\exp(60)$ to get that
$$S_3 < 4.4 \log (x+1/2).$$
This will actually be of little consequence to the final sum (as we will soon see), and so we feel no remorse in collecting here the weaker but tidier bound
\begin{equation} \label{s3}
S_3 < 5 x \log x.
\end{equation}

For $S_2$, we estimate $x/n < \alpha$ and $\Lambda(n) \leq \log x$ to get
$$S_2 < \alpha^c \log x \sum_{n=[x/\alpha]+1}^{[x]-1} \bigg| \log \frac{x}{n} \bigg|^{-1}.$$
Bounding $\Lambda(n)$ in this way is quite slack. It was pointed out to the author in private correspondence with Olivier Ramar\'{e} that one could apply the Brun--Titchmarsh theorem and save a factor of $\log x$. As we will see in Chapter 3, this would not make much of a difference in the problem of primes between consecutive cubes, and so we do not pursue this avenue.

Now, if we let $n = [x] - v$, then the problem becomes that of summing over $v=1,2,\ldots, [x]-[x/\alpha]-1$. We have
$$\bigg| \log \frac{x}{n} \bigg| = \log \frac{x}{n} > \log \frac{[x]}{n} = - \log \bigg(1- \frac{v}{[x]} \bigg) > \frac{v}{[x]}$$
and thus
$$S_2 < \alpha^c x \log x \sum_{v=1}^{[x]-[x/\alpha]-1} \frac{1}{v}.$$
One can estimate this by the known bound 
$$\sum_{n \leq x} \frac{1}{n} \leq \log x + \gamma + \frac{1}{x}$$
 where $\gamma = 0.5772\ldots$ is Euler's constant to get
\begin{equation} \label{s2}
S_2 < \alpha^c x \log x \bigg( \log(x - x/\alpha) + \gamma + \frac{1}{x-x/\alpha} \bigg).
\end{equation}

The sum $S_4$ is similar to this; we use $\Lambda(n) \leq \log(\alpha x)$ and $x/n < 1$ to get the bound
\begin{eqnarray*} 
S_4 & < & \log(\alpha x) \sum_{n =[x]+2}^{[\alpha x]} \frac{1}{\log (n/x)}.
\end{eqnarray*}
As $1<\alpha<2$, we have upon setting $n=[x]+1+v$ that
\begin{eqnarray*} 
S_4 & < & \log(\alpha x) \sum_{v=1}^{[x]} \frac{1}{\log \Big(\frac{[x]+1+v}{x}\Big)} \\
& < & \log(\alpha x) \sum_{v=1}^{[x]} \frac{1}{\log \Big(\frac{[x]+1+v}{[x]+1}\Big)} \\
& < & \log(\alpha x) \sum_{v=1}^{[x]} \frac{1}{\log \Big( 1 + \frac{v}{[x]+1}\Big)}. \\
\end{eqnarray*}
Using the estimate $\log(1+x) > 2x/3$ for $0<x<1$ we have that
\begin{eqnarray} \label{s4}
S_4 & < & \frac{3}{2} \log( \alpha x) ([x]+1) \sum_{v\leq x} \frac{1}{v}  \nonumber \\
& < &  \frac{3}{2} \log( \alpha x) ([x]+1) \bigg(\log x + \gamma + \frac{1}{x}\bigg).
\end{eqnarray}
Finally, we combine (\ref{s1s5}), (\ref{s3}), (\ref{s2}) and (\ref{s4}) to get an inequality of the form
\begin{equation}
\sum_{n=1}^{\infty} \Lambda(n) \bigg( \frac{x}{n} \bigg)^c \bigg| \log \frac{x}{n} \bigg|^{-1} < f(\alpha, x).
\end{equation}
The result follows from choosing $\alpha = 1.2$ and letting $x>e^{60}$.
\end{proof}

The immediate result of Lemma $\ref{boundbigsum}$ is that
\begin{eqnarray*}
\psi(x) & = &   \frac{1}{2 \pi i} \int_{c-i T}^{c+i T} - \frac{\zeta'(s)}{\zeta(s)} \frac{x^s}{s} ds + O^* \bigg( \frac{3.1 x \log^2 x}{\pi T} \bigg)
\end{eqnarray*}
for $x>e^{60}$, $c=1+1/\log x$ and $T>0$. We now look to shifting the line of integration so that we might involve the residues of the integrand. Let $U$ be a positive odd number and define the line segments 
$$C_1 = [c-iT,c+iT] \hspace{1in} C_2 = [c+iT,-U+iT]$$ 
$$C_{3} = [-U+iT,-U-iT] \hspace{0.8in} C_{4} = [-U-iT,c-iT]$$
and their union $C$ along with the corresponding integrals
$$I_i = \frac{1}{2 \pi i} \int_{C_i} - \frac{\zeta'(s)}{\zeta(s)} \frac{x^s}{s} ds.$$
One can note that $I_4$ is the conjugate of $I_3$, and we will later use this fact to bound both at once. We also denote by $I$ the integral around the rectangle $C$. Note that we need to account for the fact that while $T$ is stipulated not to be the ordinate of a zero of $\zeta(s)$, it might be undesirably close to such. We show in Lemma $\ref{choicet}$ that there is always some good choice of $T$ nearby, and so some work will be required later to shift our horizontal paths. Also note that any work we do in bounding $I_2$ will, by the functional equation for $\zeta(s)$, also hold for $I_{4}$ and so it follows that
\begin{equation} \label{psi}
| \psi(x) - I | < 2 |I_2|+|I_{3}| + 3.1 \frac{x \log^2 x}{\pi T}.
\end{equation}
One can use Cauchy's theorem (see Davenport \cite{davenport} for full details) to show that
\begin{equation*} 
I = x - \sum_{|\gamma| < T} \frac{x^{\rho}}{\rho} - \frac{\zeta'(0)}{\zeta(0)}+\sum_{0 < 2m < U} \frac{x^{-2m}}{2m}
\end{equation*}
where as usual $\rho = \beta+i \gamma$ denotes a nontrivial zero of $\zeta(s)$. Noting that the rightmost summation is a partial sum of the series for $\log(1-x^{-2})/2$, we can write that
$$\psi(x) = x - \sum_{|\gamma| < T} \frac{x^{\rho}}{\rho} + E(x,T,U)$$
where 
\begin{equation} \label{bigerror}
| E(x,T,U) | < \frac{\zeta'(0)}{\zeta(0)} + \frac{1}{2} \log(1-x^{-2}) + 2 |I_2|+|I_{3}| + 3.1 \frac{x \log^2 x}{\pi T}.
\end{equation}
It remains to bound $|I_2|$ and $|I_{3}|$ by deriving and making use of explicit estimates for $| \zeta'(s)/\zeta(s)|$ in appropriate regions.

We first establish a bound that holds on the lengths of the rectangle $C$ that intersect with the half-plane $\sigma \leq -1$. The stipulation that $U$ is a positive odd number is so that we might avoid the poles of $\tan \pi s/2$ which occur at the odd integers. 

\begin{lemma}  \label{bound1}
We have that
\begin{equation*}
\bigg| \frac{\zeta'(s)}{\zeta(s)} \bigg| < 9 + \log | s|
\end{equation*}
on the intersection of $C$ with $\sigma \leq -1$. 
\end{lemma}

\begin{proof}
We find it easier to bound $\zeta'(1-s)/\zeta(1-s)$ and then make the corresponding change of variable at the end. Consider the logarithmic derivative of the functional equation for $\zeta(s)$:
$$-\frac{\zeta'(1-s)}{\zeta(1-s)} = -\log 2\pi - \frac{1}{2} \pi \tan \frac{\pi s}{2} + \frac{\Gamma'(s)}{\Gamma(s)} + \frac{\zeta'(s)}{\zeta(s)}.$$
Let $\sigma \geq 2$ (so that $1-\sigma \leq -1$) and notice that $| \frac{1}{2} \pi \tan \frac{\pi s}{2}| < 2$ so long as $s$ is distanced by at least $1$ from odd integers on the real axis. We can then use
\begin{equation}  \label{gamma}
\frac{\Gamma'(s)}{\Gamma(s)} = \log s - \frac{1}{2s} - \int_0^{\infty} \frac{[u]-u+1/2}{(u+s)^2}du
\end{equation}
to bound $|\Gamma'(s)/\Gamma(s)|$ trivially. The result then follows by observing that
\begin{equation} \label{trivialzetabound}
\bigg| \frac{\zeta'(s)}{\zeta(s)} \bigg| \leq - \frac{\zeta'(2)}{\zeta(2)}  < \frac{3}{5},
\end{equation}
making the change of variable and putting it all together.
\end{proof}

We now look to the harder task of establishing a bound over the region that \textit{includes} the critical strip, as is essential for the estimation of $I_2$. 

\begin{lemma} \label{zetaaszeroes}
Let $s = \sigma +it$ where $\sigma > -1$ and $t >50$. Then
\begin{equation}
\frac{\zeta'(s)}{\zeta(s)} = \sum_{\rho} \bigg(\frac{1}{s-\rho}-\frac{1}{2+it-\rho} \bigg) + O^*(2 \log t),
\end{equation}
\end{lemma}

\begin{proof}
We start with the equation (see 12.8 of Davenport \cite{davenport})
\begin{equation} \label{zetazeroes}
-\frac{\zeta'(s)}{\zeta(s)} = \frac{1}{s-1} - B -\frac{1}{2} \log \pi + \frac{\Gamma'(\frac{s}{2}+1)}{2 \Gamma(\frac{s}{2}+1)} - \sum_{\rho} \bigg(\frac{1}{s-\rho}+\frac{1}{\rho} \bigg)
\end{equation}
where $B = \gamma/2 - 1 +\frac{1}{2} \log 4 \pi$. Successively, we set $s_0 = 2+it$ and $s = \sigma +it$ and then find the difference between the two expressions. The terms involving the $\Gamma$-function are dealt with using (\ref{gamma}), whereas the rest are estimated either trivially or with (\ref{trivialzetabound}) to arrive at the result.
\end{proof}

We aim to estimate the sum in Lemma \ref{zetaaszeroes} by breaking it into two smaller sums $S_1$ and $S_2$, where $S_1$ ranges over the zeroes $\rho = \beta+i \gamma$ with $|\gamma-t| \geq 1$ and $S_2$ is over the remaining zeroes. Note that some optimisation could be done (see Pintz \cite{pintzdisproof}) on the size of our disk but -- as we will see in the next chapter -- the marginal profits are not worth the subsequent mopping of brow.
\begin{lemma} \label{usedbybryce}
Let $s = \sigma +it$, where $\sigma > -1$ and $t >50$. Then
$$S_1 =  \sum_{|t-\gamma| \geq 1} \bigg(\frac{1}{s-\rho}-\frac{1}{2+it-\rho} \bigg)  = O^*(16 \log t).$$
\end{lemma}

\begin{proof}

We can estimate the summand as follows (see Davenport \cite[Ch. 15]{davenport}):
\begin{equation}
\bigg| \frac{1}{s-\rho} - \frac{1}{2+it-\rho} \bigg| = \frac{2 - \sigma}{| (s-\rho) (2+it-\rho)|}  < \frac{3}{(t - \gamma )^2}.
\end{equation} 
We then have that
$$S_1 < \sum_{|t - \gamma | \geq 1} \frac{3}{(t-\gamma)^2} \leq \sum_{\rho} \frac{6}{1+(t-\gamma)^2}.$$
By letting $\sigma = 2$, taking real parts in (\ref{zetazeroes}) and estimating as in the proof of Lemma \ref{zetaaszeroes} we have
$$\sum_{\rho} \text{Re} \bigg( \frac{1}{s-\rho} + \frac{1}{\rho} \bigg) <  \frac{2}{3} \log t$$
for $t > 50$. We then use the two simple facts
$$\text{Re} \bigg( \frac{1}{s-\rho} \bigg) = \frac{2 - \beta}{(2-\beta)^2+(t-\gamma)^2} > \frac{1}{4+(t-\gamma)^2}$$
and
$$\text{Re} \bigg( \frac{1}{\rho} \bigg) = \frac{\beta}{|\rho|^2} > 0$$
to get
$$\sum_{\rho} \frac{1}{4+(t-\gamma)^2} < \frac{2}{3} \log t.$$
Putting it all together we have
\begin{eqnarray*}
S_1 & < & \sum_{\rho} \frac{6}{1+(t-\gamma)^2}\\
&<& 24 \sum_{\rho} \frac{1}{4+(t-\gamma)^2} \\
& < & 16 \log t.
\end{eqnarray*}
\end{proof}

We now wish to estimate the remaining sum
$$S_2 == \sum_{|\gamma-t|<1} \bigg( \frac{1}{s-\rho} - \frac{1}{2+it - \rho} \bigg).$$
To do this, first recall that $N(T)$ denotes the number of zeroes of $\zeta(s)$ in the critical strip up to height $T$. Noting that $|2+it-\rho| > 1$, the contribution of the second term to the sum can be estimated trivially by 
$$N(t+1)-N(t-1),$$
that is, by the number of terms in the sum. Now we prove the following result.

\begin{lemma}
We have that
\begin{equation} \label{boundzerocount}
N(t+1) - N(t-1) < \log t
\end{equation}
for all $t>50$.
\end{lemma}

\begin{proof}
We can use Corollary 1 of Trudgian \cite{trudgianargument} with $T_0 = 50$ to verify that the bound holds as long as $t > 250000$. To prove it for the remaining range, we can use Odlyzko's tables \cite{odlyzko} of the zeroes of the Riemann zeta-function. A short algorithm written in Python reads in zeroes from the table and checks that the bound (\ref{boundzerocount}) holds in the remaining range. Specifically, the algorithm runs a check on the values of $t$ from 50 to 250000 in increments of $0.01$. To verify that the lemma is true for all values of $t$, we check the sharper inequality
$$N(t+1.01) - N(t-1) < \log t$$
at these discrete values and from this it follows that the result is true for all $t > 50$.
\end{proof}

At this point we have established the following result.

\begin{lemma}\label{bryce}
Let $s = \sigma +it$, where $\sigma > -1$ and $t >50$. Then
$$\frac{\zeta'(s)}{\zeta(s)} = \sum_{| t-\gamma|<1} \frac{1}{s-\rho} + O^*(19 \log t).$$
\end{lemma}

Therefore, we are concerned with bounding the magnitude of the sum over zeroes in the above lemma. Of course, the problem here is that $s$ might be close to a zero $\rho$, and this will give us significant trouble. 

At this point, we search instead for a better value of $t$ nearby, say $t_0 \in (t-1,t+1)$, which will give a workable bound. We will use this in the next section to shift our horizontal line of integration to a better height.

\begin{lemma} \label{choicet}
Let $t> 50$. There exists $t_0 \in (t-1,t+1)$ that does not depend on $\sigma$ and such that
\begin{equation} \label{bound2}
\bigg| \sum_{|\gamma-t|<1} \frac{1}{(\sigma+it_0)-\rho} \bigg| < \log^2 t + \log t.
\end{equation}
\end{lemma}

\begin{proof}
By $(\ref{boundzerocount})$, there are at most $ \log t$ zeroes in the sum. The imaginary ordinates of these zeroes partition the region of the strip into no more than ($\log t +1$) zero-free sections.  Trivially, there will always be such a section of height 
$$\frac{2}{ \log t + 1}$$
and choosing the midpoint, say $t_0$, of this section will guarantee a distance of 
$$\frac{1}{ \log t +1}$$
from any zero. As such, we have, letting $s = \sigma + i t_0$, that
$$\sum_{|\gamma-t|<1} \frac{1}{|s- \rho |} \leq \sum_{|\gamma-t|<1} \frac{1}{|\gamma - t|} \leq \sum_{|\gamma-t|<1} ( \log t +1) \leq   \log^2 t +  \log t.$$
\end{proof}

Finally, we can put the previous three lemmas together to get the following.

\begin{lemma}
Let $\sigma > - 1, t > 50$. Then there exists $t_0 \in (t-1,t+1)$ such that for every $\sigma > -1$ we have
$$\bigg| \frac{\zeta'(\sigma+it_0)}{\zeta(\sigma+it_0)} \bigg| < \log^2 t+ 20\log t.$$
\end{lemma}

That is, if our contour is somewhat close to a zero, we can shift it slightly to a line where we have good bounds.

We now bound the error term $E(x,T,U)$ in $(\ref{bigerror})$, by using our bounds and estimating each integral trivially. Using Lemma \ref{bound1}, we have
\begin{eqnarray*}
|I_3| & = &\frac{1}{2\pi} \bigg| \int_{-U-iT}^{-U+iT} - \frac{\zeta'(s)}{\zeta(s)} \frac{x^s}{s} ds \bigg|\\ \\ 
& < &  \int_{-T}^{T} \frac{9 + \log \sqrt{U^2+T^2}}{2 \pi x^U T}dt \\ \\
&=&   \frac{18+2\log \sqrt{U^2+T^2}}{2 \pi x^U}.
\end{eqnarray*}
We save this, for soon we will combine our estimates and bound them in unison upon an appropriate choice for $U$. Consider now the problem of estimating $I_2$, and the issue that $T$ might be close to the ordinate of a zero. From Lemma \ref{choicet}, there exists some $T_0 \in (T-1,T+1)$ that we should integrate over instead. We thus aim to shift the line of integration from $C_2$ to 
$$C_2 ' = [-U+iT_0, c +iT_0].$$

It follows from Cauchy's theorem that
$$|I_2| < \sum_{T-1 < \gamma < T+1} \bigg| \frac{x^{\rho}}{\rho} \bigg| + |I_5| + |I_6| +|I_7|+|I_8|$$
where
\begin{eqnarray*}
I_5 & = & \frac{1}{2 \pi i} \int_{-U+iT}^{-U+iT_0} - \frac{\zeta'(s)}{\zeta(s)} \frac{x^s}{s} ds \hspace{0.7in} I_6  =   \frac{1}{2 \pi i}  \int_{-U+iT_0}^{-1+iT_0} - \frac{\zeta'(s)}{\zeta(s)} \frac{x^s}{s} ds\\ \\
I_7 & = &  \frac{1}{2 \pi i}  \int_{-1+iT_0}^{c+iT_0} - \frac{\zeta'(s)}{\zeta(s)} \frac{x^s}{s} ds \hspace{0.8in} I_8  =   \frac{1}{2 \pi i}  \int_{c+iT}^{c+iT_0} - \frac{\zeta'(s)}{\zeta(s)} \frac{x^s}{s} ds.
\end{eqnarray*}
From $(\ref{boundzerocount})$, we can estimate the sum by
$$\sum_{T-1 < \gamma < T+1} \bigg| \frac{x^{\rho}}{\rho} \bigg| < \sum_{T-1 < \gamma < T+1} \frac{x}{T-1} < \frac{2 x \log T}{T-1}.$$

We can bound $I_5$ in the same way as $I_3$ to obtain
$$|I_5| <\frac{18+2\log \sqrt{U^2+(T+1)^2}}{2 \pi x^U T}.$$
Bounding $I_6$ is done using Lemma \ref{bound1}:
$$|I_6| < \frac{9+\log \sqrt{U^2+(T+1)^2}}{2 \pi x (T-1)}$$
We also use Lemma \ref{choicet} to get
$$|I_7| < \frac{e}{2 \pi (T-1)} (\log^2 (T+1) + \log (T+1)).$$
To get an upper bound for $I_8$, we note that $c = 1+ 1/\log x$ and then (\ref{fordlemma}) gives us that
$$\bigg| \frac{\zeta'(s)}{\zeta(s)} \bigg| \leq \bigg| \frac{\zeta'(c)}{\zeta(c)} \bigg| < \log x.$$ 
Then, estimating trivially gives us that
$$|I_8| < \frac{ex \log x}{\pi(T-1)}.$$

Throwing all of our estimates for the terms in $(\ref{bigerror})$ together, implanting the information that $T \leq x$, $x > e^{60}$ and letting $U$ be equal to the odd integer closest to $x$ we obtain Theorem \ref{explicitformula}. In the next chapter, we will apply this result to the problem of primes between cubes.


\chapter{Primes in Short Intervals}
\label{chapter3}

\begin{quote}
\textit{``That's the reason they're called lessons,'' the Gryphon remarked: ``because they lessen from day to day.''}
\end{quote}

\section{Primes between consecutive cubes}

Legendre's conjecture is the assertion that there is at least one prime between any two consecutive squares. Confirmation of this seems to be out of reach, for applying modern techniques on the assumption of even the Riemann hypothesis does not suffice in forming a proof (see Davenport \cite{davenport} for a discussion). It is thus the aim of this first section to study the weaker problem of primes between cubes, where some progress has already been made.

Consider first the more general problem of showing the existence of at least one prime in the interval $(x,x+x^{\theta})$ for some $\theta \in (0,1)$ and for all sufficiently large $x$. These are \textit{short} intervals; generally, any interval of the form $(x,x+h(x))$ is said to be short if $h/x \rightarrow 0$ as $x \rightarrow \infty$.

In 1930, Hoheisel \cite{hoheisel} was able to solve the problem for $\theta = 1-1/33000$, that is, that there will be a prime in the interval
$$(x,x+x^{32999/33000})$$
for all sufficiently large $x$. He did this by using the Riemann--von Mangoldt explicit formula in conjunction with an appropriate zero-free region and zero-density estimate. Using Hoheisel's ideas, Ingham \cite{ingham} was able to prove a more general theorem, specifically that if one has a bound of the form
$$\zeta(1/2+it) = O(t^c )$$
for some $c>0$, then one can take
$$\theta = \frac{1+4c}{2+4c}+\epsilon.$$
Notably, the bound for $\zeta(1/2+it)$ can be used to construct a zero-density estimate, and this in turn furnishes a value for $\theta$ through a subsequent application of the explicit formula. Hardy and Littlewood were able to give a value of $c=1/6+\epsilon$, which corresponds to $\theta = 5/8+\epsilon$. From this, one sets $x = n^3$ and it follows that for all sufficiently large $n$ there exists a prime in the interval
$$(x,x+x^{5/8}) = (n^3,n^3+n^{15/8+ \epsilon}) \subset (n^3, (n+1)^3).$$
That is, with finitely many exceptions, there is a prime between any two consecutive cubes. The reader should, however, note that consideration of the interval
$$(x,x+3x^{2/3})$$
is sufficient for primes between cubes and as such this is the interval we use throughout this section. Expanding the expression $(n+1)^3$ shows that we could use the slightly larger interval
$$(x,x+3x^{2/3}+3 x^{1/3} + 1),$$
however, the difference is negligible for the large values of $x$ in which we deal.

The purpose of this section is to make explicit the result on primes between consecutive cubes, in that we determine a numerical lower bound beyond which this result holds. By working through the paper of Ingham \cite{ingham}, we can do this thanks to Ford's \cite{ford} explicit zero-free region, Ramar\'{e}'s \cite{ramare} explicit zero-density theorem and our \textit{actually} explicit formula (see Theorem \ref{explicitformula}). We shall bring these together to prove our main theorem.

\begin{theorem} \label{primecube}
There is a prime between $n^3$ and $(n+1)^3$ for all $n \geq \exp(\exp(33.3))$.
\end{theorem}

We should note that a result has been given by Cheng \cite{cheng}, in which he purports to prove Theorem \ref{primecube} for the range $n \geq \exp(\exp(15))$. We should note, however, that he establishes the result for $n^3 \geq \exp(\exp(45))$ and then incorrectly infers that $n \geq \exp(\exp(15))$. There are some other errors also, notably in the proof of Theorem 3 in his paper \cite{cheng}, the first inequality sign is backwards and he has used a sum over prime powers instead of the appropriate sum over primes.

Before we launch into our proof of Theorem \ref{primecube}, we will prove a result from the other direction. The following lemma explains our earlier optimism in developing Theorem \ref{explicitformula} for the range~$x>e^{60}$.

\begin{lemma}
There is a prime in the interval $(x,x+3x^{2/3})$ for all $x \leq e^{60}$.
\end{lemma}

\begin{proof}
Theorem 2 in Ramar\'{e} and Saouter's paper \cite{ramaresaouter} states that there is a prime in the interval 
$$(x(1-\Delta^{-1}),x]$$
for all $x > x_0$ and where $\Delta$ is a function of $x_0$ as given in Table 1 of their paper. It is a straightforward task to use this table to verify that there is a prime in $(x,x+3x^{2/3})$ for all $x \leq e^{60}$. One simply works through the table whilst checking that the Ramar\'{e}--Saouter interval is contained in the short interval. 
\end{proof}

Finally, we should acknowledge the striking result of Baker, Harman and Pintz \cite{bakerharmanpintz}, that the interval $(x,x+x^{0.525})$ contains a prime for all sufficiently large $x$. This is tantalisingly close to $\theta = 1/2$, which would furnish a proof of Legendre's conjecture with at most finitely many exceptions. The authors also note that their result is effective, though furnishing an explicit result would surely make for quite a large project.

To begin the proof of Theorem \ref{primecube}, we define the Chebyshev $\theta$-function as
$$\theta(x) = \sum_{p \leq x} \log p.$$
This is similar to Chebyshev's $\psi$-function, though we have removed the contribution owing to the powers of primes. Consider that 
$$\theta_{x,h} = \theta(x+h)-\theta(x) = \sum_{x < p \leq x+h} \log p$$
is positive if and only if there is at least one prime in the interval $(x,x+h]$. Many questions involving the primes can be phrased in terms of $\theta_{x,h}$. For example, the twin prime conjecture --  there are infinitely many primes $p$ such that $p+2$ is also prime -- is equivalent to $\theta_{p,2}$ taking on a positive value infinitely often where $p$ is a prime. 

Clearly, we set $h = 3 x^{2/3}$ to tackle the problem of primes between cubes. Then, if $\theta_{x,h} > 0$ for all $x > x_0$, there is a prime in the interval $(x,x+3x^{2/3}]$ for all $x > x_0$. If we then set $x = n^3$, we have that there is a prime in the interval $(n^3, (n+1)^3]$ for all integers $n > n_0 = x_0^{1/3}$. It is our intention to determine explicitly a value for $x_0$ and thus $n_0$. 

We call on our explicit formula. Substituting $(x+h)$ and then $x$ into Theorem \ref{explicitformula} and taking the difference, we find that: 
\begin{equation} \label{psiinterval}
\psi(x+h) - \psi(x) > h - \bigg| \sum_{|\gamma| < T} \frac{(x+h)^{\rho} - x^{\rho}}{\rho}\bigg| - \frac{4 (x+h) \log^2 (x+h)}{T}.
\end{equation}
Whilst the above will tell us information about prime powers, we are interested in primes. We thus require the following lemma, which is a combination of Proposition 3.1 of Dusart \cite{dusart} and Corollary 2 of Platt and Trudgian \cite{platttrudgian}.

\begin{lemma} \label{psitheta}
Let $x \geq 121$. Then
$$0.9999  x^{1/2} < \psi(x) - \theta(x) < (1+7.5 \cdot 10^{-7}) x^{1/2} + 3 x^{1/3}.$$ 
\end{lemma}

An application of this lemma to (\ref{psiinterval}) gives us that
\begin{eqnarray} \label{big}
\theta_{x,h} & > & h - \bigg| \sum_{|\gamma| < T} \frac{(x+h)^{\rho} - x^{\rho}}{\rho}\bigg| - \frac{4 (x+h) \log^2 (x+h)}{T}  \nonumber \\ 
&-& (1+7.5 \cdot 10^{-7}) (x+h)^{1/2}-3(x+h)^{1/3} + 0.9999 x^{1/2}.
\end{eqnarray}
It remains to estimate the sum in this inequality, choose $T=T(x)$ appropriately and then find $x_0$ such that $\theta_{x,h}$ is positive for all $x>x_0$. We define
$$S = \bigg| \sum_{|\gamma| < T} \frac{(x+h)^{\rho} - x^{\rho}}{\rho}\bigg|.$$
It follows that
\begin{eqnarray*}
S & = & \bigg| \sum_{|\gamma| < T} \int_{x}^{x+h} t^{\rho-1} \bigg| \leq  \sum_{|\gamma| < T} \int_{x}^{x+h} t^{\beta-1}  \leq h \sum_{| \gamma| < T} x^{\beta-1}.
\end{eqnarray*}
Now, the identity
\begin{eqnarray*}
\sum_{|\gamma| < T} ( x^{\beta-1} - x^{-1}) & = & \sum_{|\gamma| < T} \int_0^{\beta} x^{\sigma - 1} \log x \ d \sigma \\
& = & \int_0^1 \sum_{ \beta> \sigma, |\gamma| < T } x^{\sigma-1} \log x \ d \sigma
\end{eqnarray*}
can be reformulated as follows:
\begin{equation} \label{summypoos}
\sum_{|\gamma| < T} x^{\beta  -1} = 2 x^{-1} N(T) + 2 \int_0^{1} N(\sigma,T) x^{\sigma-1} \log x \ d \sigma,
\end{equation}
where $N(T)$ and $N(\sigma, T)$ are as defined in Chapter 2.

We can estimate the above sum, and thus $S$, with the assistance of some explicit bounds. Firstly note, that by Corollary 1 of Trudgian \cite{trudgianargument} we have that
\begin{equation}\label{trudgiancount}
N(T) < \frac{T \log T}{2 \pi}
\end{equation}
for all $T>15$. Explicit estimates for $N(\sigma,T)$ are rare, though have come to light recently through the likes of Kadiri \cite{kadiri} and Ramar\'{e} \cite{ramare}, who have produced zero-density estimates of rather different shape to each other. Ramar\'{e}'s estimate, which is an explicit and asymptotically better version of Ingham's \cite{ingham} original density estimate, is required for the problem of primes between cubes. We give the result here, which is a corollary of Theorem 1.1 of \cite{ramare}.

\begin{lemma} \label{ramaredensity}
Let $T \geq 2000$ and $\sigma \geq 0.52$. Then
$$N(\sigma,T) \leq 9.7 (3T)^{8(1-\sigma)/3} \log^{5-2\sigma} T + 103 \log^2 T.$$
\end{lemma}

There are some comments to make here. Firstly, it should be noted that when $\sigma > 5/8$, the zero-density estimate in Lemma \ref{ramaredensity} is stronger than that of Kadiri's for all sufficiently large $T$. For the problem of primes between cubes, this is precisely the range of $\sigma$ we are most concerned with, and $T$ will certainly be taken large enough so that our choice of zero-density estimate is indeed the best one possible.

Also, consider a more general zero-density estimate of the form
$$N(\sigma,T) \leq B T^{K(1-\sigma)} \log^L T.$$
Note that we could apply some trivial bounds to Ramar\'{e}'s estimate so that it takes such a form. The most important player here is the constant $K$, which (see Theorem 1 of Ingham \cite{ingham}) implies that one can take $\theta = 1-K^{-1}$, so long as we have the zero-free region (\ref{vinogradovkorobov}) or, rather, anything of greater order than the so-called Littlewood zero-free region (see Theorem 5.17 of Titchmarsh \cite[Ch.\ 5]{titchmarsh1986theory})
$$\frac{c \log \log |\gamma|}{\log | \gamma |} \leq \beta \leq 1-  \frac{c \log \log |\gamma|}{\log | \gamma |}.$$

For the proof of Theorem \ref{primecube}, it would be useful to insert a completely general zero-density estimate and zero-free region (with unspecified constants). This would allow us to obtain a function which takes in the important values and returns a result for the primes-between-cubes problem. Really, this is how most mathematics is surely done, for it allows one to see immediate changes without repeating the labours of other mathematicians. However, keeping everything general disables a lot of the freedom we require when working explicitly. It is very easy to say something like
$$\log x + \log^2 x < 2 \log^2 x$$
for all $x\geq 3$ but, on the other hand, finding a suitable range so that 
$$\log^a x + \log^b x < c \log^b x$$
is not as straightforward.

As such, for the zero-density theorem, we will only leave one constant unspecified. We will conduct our working with $A$ in place of the constant $9.7$ in Lemma \ref{ramaredensity}. The reason for this is simple: if one were to sharpen up the numbers elsewhere in the estimate, it is always straightforward enough to express this as a change in the number-out-the-front.

The following zero-free region, an explicit form of the Vinogradov--Korobov region as derived by Ford \cite{ford}, will also be required.

\begin{lemma} \label{fordregion}
Let $T \geq 3$. Then there are no zeroes of $\zeta(s)$ in the region given by $\sigma \geq 1 - \nu(T)$ where
$$\nu(T) = \frac{1}{57.54 \log^{2/3} T (\log \log T)^{1/3}}.$$
\end{lemma}

For the same reasons as with the zero-density estimate, we will leave the scaling constant general, and so we work with $c$ in place of the $57.54$ in the above lemma. 

We split the integral in (\ref{summypoos}) into two parts, one over the interval $0 \leq \sigma \leq 5/8$, where $N(\sigma,T)$ may as well be bounded by $N(T)$, and another over $5/8 \leq \sigma \leq 1-\nu(T)$. By applying the relevant estimates in the above two lemmas, we get
\begin{eqnarray} \label{crackers}
\sum_{|\gamma| < T} x^{\beta-1} & < & 2 x^{-1} N(T) + 2 x^{-1} N(T) \log x \int_0^{5/8} x^{\sigma} dx \nonumber  \\
& + & 2 A x^{-1} (3T)^{8/3} \log x \log^5 T \int_{5/8}^{1-\nu(T)} \bigg( \frac{x}{(3T)^{8/3} \log^2 T} \bigg)^{\sigma} \ d \sigma  \nonumber \\
& + & 103 x^{-1} \log x \log^2 T \int_{5/8}^{1-\nu(T)} x^{\sigma} \ d \sigma.
\end{eqnarray}

The working out is routine, yet tedious. We give the qualitative details to the extent that the reader can follow the process. We introduce the parameter $k \in (\frac{2}{3},1)$, which will play a part in the relationship between $T$ and $x$. The reasons for the range of values of $k$ will become clear soon. We let $T=T(k,x)$ be the solution to the equation
\begin{equation} \label{gumbie}
\frac{x}{(3T)^{8/3} \log^2 T} = \exp( \log^k x).
\end{equation}
We then have that
$$x =  \exp( \log^k x) (3T)^{8/3} \log^2 T > T^{8/3}.$$
Upon performing the integration in (\ref{crackers}), we directly apply the equation (\ref{gumbie}), along with the bound for $N(T)$ and the fact that $\log T < (3/8) \log x$, to get
\begin{eqnarray} \label{dong}
\sum_{|\gamma| < T} x^{\beta-1} & < & \frac{ e^{-\frac{3}{8} \log^k x} \log^{1/4} x}{3^{3/4} 8^{1/4} \pi}+ \frac{27A}{256} \log^{4-k} x ( e^{-\nu(T) \log^k x}- e^{-(3/8) \log^k x}) \nonumber \\ 
& + & \frac{927 A}{32} \log^2 x (e^{- \nu(T) \log x} - x^{-3/8}).
\end{eqnarray}
There is some cancellation in the above inequality. First, we need to estimate one of the exponential terms involving $\nu(T)$. We have that
\begin{eqnarray*}
e^{-v(T) \log x} & = & \exp\bigg( -  \frac{\log x}{c \log^{2/3} T (\log \log T)^{1/3}} \bigg) \\
& < &  \exp\bigg( -  \frac{4}{3^{2/3} c} \bigg( \frac{\log x}{  \log \log x} \bigg)^{1/3} \bigg).
\end{eqnarray*}
Now, upon expansion of (\ref{dong}) and using the above we can notice that
$$- \frac{27 A}{256} (\log x)^{4-k} e^{- (3/8) \log^k x}+ \frac{927 A}{32} \log^2 x (e^{- \nu(T) \log x} - x^{-3/8}) < 0.$$
This is clear if one looks at the dominant terms, but one can verify this by dividing through by $A$, stipulating that $c \leq 57.54$, $k \in (\frac{2}{3},1)$, and taking $x > e^{60}$. It follows that
\begin{eqnarray} 
\sum_{|\gamma| < T} x^{\beta-1} & < & \frac{ e^{-\frac{3}{8} \log^k x} \log^{1/4} x}{3^{3/4} 8^{1/4} \pi}+ \frac{27A}{256} (\log x)^{4-k} e^{-\nu(T) \log^k x}.
\end{eqnarray}
The remaining exponential term involving $\nu(T)$ is dealt with as before to get
\begin{eqnarray*} 
S \leq h \sum_{|\gamma| < T} x^{\beta-1} & < & \frac{ h e^{-\frac{3}{8} \log^k x} \log^{1/4} x}{3^{3/4} 8^{1/4} \pi}+ \frac{27Ah}{256} (\log x)^{4-k} \exp\bigg(- \frac{4}{3^{2/3}c}  \frac{\log^{k-2/3} x}{(\log \log x)^{1/3}} \bigg).
\end{eqnarray*}

Therefore, we may write $(\ref{big})$ as
$$\theta_{x,h} > h - f(x,h,k,A,c) - g(x,h,k)-E(x,h,k)$$
where
\begin{eqnarray*}
f(x,h,k,A,c) & = & \frac{27 A h}{256} (\log x)^{4-k} \exp\bigg( - \frac{4}{3^{2/3} c} \frac{\log^{k-2/3} x}{(\log \log x)^{1/3}}\bigg),\\ \\
g(x,h,k) & = & 12 \bigg(\frac{3}{8} \bigg)^{3/4} \frac{(x+h) \log^{11/4} (x+h)}{x^{3/8}} \exp\bigg(\frac{3}{8} \log^k x\bigg),\\ \\
 E(x,h,k) & = & - \frac{ h (\log x)^{1/4} \exp(-\frac{3}{8} \log^k x)}{6^{3/4} \pi}-(1+7.5 \cdot 10^{-7})(x+h)^{1/2} \\\
 &-&3(x+h)^{1/3}+0.9999 x^{1/2}.
\end{eqnarray*}

First, we look to bound the error. Noting that $x > e^{60}$ and $h = 3x^{2/3}$, we use the fact that $k=2/3$ will give us the worst possible error to get
$$\frac{E(x,3 x^{2/3},2/3)}{3x^{2/3}} < 10^{-3}.$$
Thus, one can show that positivity holds if the following two inequalities are simultaneously satisfied:
\begin{enumerate}
\item $f(x,h,k,A,c) < \frac{1}{2} (1-10^{-3}) h$,
\item $g(x,h,k) < \frac{1}{2} (1-10^{-3}) h.$
\end{enumerate}
This splitting simplifies our working greatly whilst perturbing the solution negligibly. To be convinced of this, one could consider the right hand side of each of the above inequalities as being equal to $h$, in some better-than-possible scenario. It turns out that the improvements would hardly be noticeable.

Now, in the first inequality, we take the logarithm of both sides and set $x = e^y$ to get
\begin{equation}\label{1}
\log\bigg(\frac{27 A}{256}\bigg) + (4-k) \log y - \frac{4}{3^{2/3} c} \frac{y^{k-2/3}}{\log^{1/3} y} < \log \bigg(\frac{1}{2} (1-10^{-3})\bigg).
\end{equation}
This is easy to solve using \textsc{Mathematica}, given knowledge of $A$, $k$ and $c$. There are some notes to make here first. We can see that $A$, the constant in front of Ramar\'{e}'s zero-density estimate has little contribution, for being in the argument of the logarithm. On the other hand, $c$ plays a much larger part from where it is positioned. We can also see the reason for $k > 2/3$: this guarantees a solution to the above inequality.

We deal with the second inequality in the same way, but first we notice that
$$\frac{g(x,h,k)}{h} < \frac{2 \log^{11/4} x}{x^{1/24}} \exp\bigg( \frac{3}{8} \log^k x\bigg).$$
This is obtained using the main result of Ramar\'{e} and Saouter \cite{ramaresaouter}, namely that
$$x+h < \frac{x}{1-\Delta^{-1}}$$
where $\Delta = 28 314 000$ as given in their paper. Thus, using the same approach as before we get
\begin{equation} \label{2}
\frac{11}{4} \log y + \frac{3}{8} y^k-\frac{1}{24} y < \log\bigg( \frac{1}{4} (1-10^{-3})\bigg).
\end{equation}

We notice here our reason for having $k<1$. One can also see the reason for leaving $k$ free to vary in $(\frac{2}{3},1)$. There should be an optimal value of $k$, where the solution range of the above two inequalities are equal and their intersection is minimised. 

No rigorous analysis needs to be conducted; we set $A=9.7$, $c=57.54$ and use the \texttt{Manipulate} function of \textsc{Mathematica} to ``hunt'' for a good value of $k$. It turns out that upon choosing $k=0.9359$, we have that both inequalities are satisfied for $y>8 \times 10^{14}$, or $x^{1/3} > \exp(\exp(33.217))$, which proves our main result. 

\section{Improvements and higher powers}

In this short section, we will discuss how improvements to the parameters $A$ and $c$ will affect our result for primes between cubes. We will then consider the problem of primes between higher integer powers.

Let us consider first improving the zero-density estimate given by Ramar\'{e}. Say, for the sake of discussion, one could obtain a value of $A=10^{-4}$. Then we would obtain our result instead with $n \geq \exp(\exp(32.7))$, an improvement which would probably not be worth the efforts required to obtain such a value of $A$. On the other hand, reducing the exponent of $T$ in Ramar\'{e}'s estimate would correspond to incredibly small values of $A$, seeing as we take $T$ to be arbitrarily large. Therefore, it would be useful for a zero-density estimate of smaller order to be made explicit for use in this problem. One can see Titchmarsh \cite[Ch.\ 9]{titchmarsh1986theory} for some discussion of these.

Changes in the constant $c$ are more effective, though seemingly much more difficult to obtain. A value of $c=40$ would yield only $n \geq \exp(\exp(31.88))$, and $c=20$ would give $n \geq \exp(\exp(29.6))$. The removal of the $(\log \log T)^{1/3}$ would give a similar result to this. 

There are other parameters where one might wish to direct future efforts. In Ramar\'{e}'s zero density estimate, one might consider the power $5-2\sigma$ of the logarithm to be $L-2\sigma$. The main difference in our working would be $(L-1-k)$ in place of $(4-k)$ in the reduced form of our second inequality. The following table summarises the improvements which would follow, namely the existence of a prime between $n^3$ and $(n+1)^3$ for all $n \geq n_0$. 
\begin{center}
  \begin{tabular}{ | c | c | }
    \hline
    $L$  & $\log \log n_0$ \\ \hline \hline
    5 & 33.217 \\ \hline
    4 & 31.8 \\ \hline
    3 & 29.8 \\ \hline
    2 & 22.19 \\ \hline
    
  \end{tabular}
\end{center}

Turning now to the error term of Theorem \ref{explicitformula} one could also consider a smaller constant in place of $2$. This constant, however, would appear in the logarithm of the right hand side of (\ref{2}), and thus make little difference. On this note, Wolke \cite{wolke} has derived the Riemann--von Mangoldt explicit formula with an error term which is
$$O\Big(\frac{x \log x}{T \log (x/T)}\Big) $$
and thus $O(x/T)$ for the choice of $T(x)$ used for our problem. One may propose all sorts of ``good'' explicit constants for the above error term and try them via the methods of this paper, but there will be no major improvements.

Thus, really, one expects a major result, or perhaps the collaboration of minor ones, to make significant progress on this problem.

In lieu of a complete result on the problem of primes between cubes, we consider instead primes between $m$th powers, where $m$ is some positive integer. Appropriately, we choose $h = m x^{1-1/m}$, and we are able to prove the following result.

\begin{theorem} \label{mpowers}
Let $m \geq 4.971 \times 10^9$. Then there is a prime between $n^m$ and $(n+1)^m$ for all $n \geq 1$.
\end{theorem}

The result seems absurd on a first glance as the value of $m$ is quite large. We shall leave it to others to attempt to bring the value down. We prove the above theorem as follows: for our choice of $h$, it follows that (\ref{2}) becomes
\begin{equation} \label{new2}
\frac{11}{4} \log y - \Big( \frac{3}{8} - \frac{1}{m} \Big) y+\frac{3}{8} y^k < \log\Big( \frac{m}{12} (1-10^{-3})\Big)
\end{equation}
whereas (\ref{1}) remains the same. As before, we can, for some given $m$, choose $k$ and find $n_0$ such that there is a prime between $n^m$ and $(n+1)^m$ for all $n \geq n_0$ by solving both inequalities. Some results are given in the following table.
\begin{center}
  \begin{tabular}{ | c | c | c | }
    \hline
    $m$  & $k$ & $\log \log n_0$ \\ \hline \hline
    4 & 0.9635 & 29.240 \\ \hline
    5 & 0.9741 & 27.820 \\ \hline
    6 & 0.9796 & 27.230 \\ \hline
    7 & 0.983 & 26.427 \\ \hline
    1000 & 0.9998 & 19.807 \\ \hline
  \end{tabular}
\end{center}

One can see that this method has its limitations, even in the case of higher powers. In any case, we have that there is a prime in $(n^{1000}, (n+1)^{1000})$ for all $n  \geq \exp(\exp(19.807))$. It follows that, for $m \geq 1000$, there is a prime between $n^m$ and $(n+1)^m$ for all 
\begin{equation} \label{lower}
n \geq \exp\bigg( \frac{1000 \exp(19.807)}{m}\bigg).
\end{equation}
We could choose $m = 1000 \exp(19.807)$ which is approximately $4 \times 10^{11}$ to get primes between $n^m$ and $(n+1)^m$ for all $n \geq e$. Bertrand's postulate then improves this to all $n \geq 1$.

However, we can use Corollary 2 of Trudgian \cite{trudgianpomerance} to improve on this value of $m$. This states that for all $x \geq 2 \ 898 \ 242$ there exists a prime in the interval 
$$\bigg[ x,x\bigg(1+\frac{1}{111\log^2 x}\bigg) \bigg].$$
If we set $x = n^m$, we might ask when the above interval falls into $[n^m,n^m+m n^{m-1}]$. One can rearrange the inequality 
$$n^m \bigg(1+\frac{1}{111\log^2 (n^m)}\bigg) < n^m+m n^{m-1}$$
to get
\begin{equation} \label{upper}
\frac{n}{\log^2 n} < 111 m^3.
\end{equation}
We wish to choose the lowest value of $m$ for which the solution sets of (\ref{lower}) and (\ref{upper}) first coincide. It is not to hard to see that this equates to solving simultaneously the equations
\begin{equation*}
n = \exp\bigg( \frac{1000 \exp(19.807)}{m}\bigg)
\end{equation*}
and
\begin{equation*} 
\frac{n}{\log^2 n} = 111 m^3.
\end{equation*}
We do this by substituting the first equation directly into the second to get
$$\exp\bigg( \frac{1000 \exp(19.807)}{m}\bigg) = 111 (1000 \exp(19.807))^2 m$$
which can easily be solved with \textsc{Mathematica} to prove Theorem \ref{mpowers}. 

A gambit was suggested to the author, that he might consider the larger interval $[n^m, (n+1)^m]$ in favour of $[n^m, n^m+m n^{m-1}]$. The improvements, however, are negligible. Using the well-known bound
$$\binom{n}{k} \leq \frac{n^k}{k!},$$
we have from the binomial theorem that
\begin{eqnarray*}
(n+1)^m - (n^m +m n^{m-1})  & = & \sum_{k=2} \binom{m}{k} n^{m-k} \\
& \leq & m n^{m-1} \sum_{k=2}^{m} \frac{1}{k!} \bigg( \frac{m}{n} \bigg)^{k-1}  \\
& < & m n^{m-1} \bigg( \frac{m}{n} \bigg) \sum_{k=2}^{m} \frac{1}{k!}.
\end{eqnarray*}
Unfortunately, from (\ref{lower}) we have that $m/n$ is extremely small for the values of $m$ which improve on Theorem \ref{mpowers}.

\section{A result on the Riemann hypothesis}

We have seen that without the Riemann hypothesis, explicit results are obtainable using similar-natured zero-free regions and zero-density estimates. If we assume this hypothesis, however, we replace such estimates with the simple identity $\beta=~1/2$, and subsequently our working and results become far neater.

On the assumption of the Riemann hypothesis, von Koch \cite{vonkoch} proved that there exists a constant $k$ such that the interval $(x-k \sqrt{x} \log^2 x, x)$ contains a prime for all $x \geq x_0$, where $x_0$ is sufficiently large. This is often written as
$$p_{n+1} - p_n = O(p_n^{1/2} \log^2 p_n),$$
where $p_n$ denotes the $n$th prime. Schoenfeld \cite{schoenfeld} made this result  explicit, showing that one can take $k = 1/(4 \pi)$ and $x_0 = 599$. 

Cram\'{e}r \cite{cramer} sharpened the result of von Koch by proving the following theorem.

\begin{theorem} \label{cramer}
Suppose the Riemann hypothesis is true. Then it is possible to find a positive constant $c$ such that
\begin{equation} \label{cramer1}
\pi(x+c \sqrt{x} \log x) - \pi(x) > \sqrt{x}
\end{equation}
for $x \geq 2$. Thus if $p_n$ denotes the $n$th prime, we have
\begin{equation} \label{cramer2}
p_{n+1} - p_n = O(\sqrt{p_n} \log p_n).
\end{equation}
\end{theorem}
Goldston \cite{goldston} made the above result more precise by showing that one could take $c=5$ in (\ref{cramer1}) for all sufficiently large values of $x$. He also showed that 
$$p_{n+1} - p_n < 4 p_n^{1/2} \log p_n$$
for all sufficiently large values of $n$. It should be noted that Goldston was not trying in any way to find the optimal constants; he was providing a new proof of Cram\'{e}r's theorem. Ramar\'{e} and Saouter \cite{ramaresaouter} furthered this line of research by showing that for all $x \geq 2$ there exists a prime in the interval $(x-\frac{8}{5} \sqrt{x} \log x, x]$.

The purpose of this section is to give the following improvement on the work of Ramar\'{e} and Saouter.

\begin{theorem} \label{one}
Suppose the Riemann hypothesis is true. Then there is a prime in the interval $(x-\frac{4}{\pi} \sqrt{x} \log x,x]$ for all $x\geq 2$.
\end{theorem}

We prove this theorem using a weighted version of the Riemann--von Mangoldt explicit formula and some standard explicit estimates for sums over the zeroes of the Riemann zeta-function. It should be noted that the constant $4/\pi = 1.273\ldots$ appearing in Theorem \ref{one} is not optimal. In fact, it is unclear what the limitations of the method are, though we can compare our result with Cram\'{e}r's conjecture \cite{cramerorder} that
$$p_{n+1}-p_n = O(\log^2 p_n).$$
We will show that with a bit more consideration, the constant $4/\pi$ can be improved as follows.

\begin{theorem} \label{implicit}
Suppose the Riemann hypothesis is true and let $\epsilon > 0$. Then there is a prime in the interval $(x-(1+\epsilon) \sqrt{x} \log x,x]$ provided that $x$ is sufficiently large.
\end{theorem}

It is not clear to the author whether the optimal constant is $1$ or something less, though it would be interesting to see whether any improvements are readily forthcoming. In a paper with Greni\'{e} and Molteni \cite{dudekmoltenigrenie}, we show that a far more sophisticated method yields the same result as Theorem \ref{implicit}. Moreover, we actually make Theorem \ref{implicit} explicit in the following way.

\begin{theorem}\label{th:A1}
Suppose the Riemann hypothesis is true. Let $x\geq 2$ and 
$$d = \frac{1}{2} + \frac{2}{\log x}.$$
Then there is a prime in the interval $(x-c\sqrt{x}\log x,x+c\sqrt{x}\log x)$ and at least $\sqrt{x}$ primes in $(x-(c+1)\sqrt{x}\log
x,x+(c+1)\sqrt{x}\log x)$.
\end{theorem}

We will not prove this result here, as the details are somewhat similar to the proof of Theorem \ref{implicit}. Notably, it follows from Theorem \ref{implicit} that Theorem \ref{cramer} can be taken with $c = 3 + \epsilon$ for sufficiently large values of $x$. It is also clear from the Prime Number Theorem that $c>1$.

Finally, the reader may wish to see the work of Heath-Brown and Goldston \cite{goldstonheathbrown}, for they show that one has an arbitrarily small constant in place of the $4/\pi$ in Theorem \ref{one} on the assumption of some more sophisticated conjectures.

We will now prove Theorems \ref{one} and \ref{implicit}. In consideration of the von Mangoldt function, we define the weighted sum (see Ingham \cite[Ch.\ 4]{inghambook} for more details)
$$\psi_1 (x) = \sum_{n \leq x} (x-n) \Lambda(n) = \int_2^x \psi(t) dt$$
and first prove an analogous explicit formula.

\begin{lemma} \label{first}
For $x>0$ and $x \notin \mathbb{Z}$ we have
\begin{equation} \label{explicit}
\psi_1(x) = \frac{x^2}{2} - \sum_{\rho} \frac{x^{\rho+1}}{\rho (\rho +1)} - x \log(2\pi) + \epsilon(x)
\end{equation}
where
$$|\epsilon(x)| < \frac{12}{5}.$$
\end{lemma}

\begin{proof}
We integrate both sides of (\ref{RVM}) over the interval $(2,x)$ to get
$$\psi_1(x) = \frac{x^2}{2} - \sum_{\rho} \frac{x^{\rho+1}}{\rho (\rho +1)} - x \log(2\pi) + \epsilon(x)$$
where
$$|\epsilon(x)| < 2 + \bigg| \sum_{\rho} \frac{2^{\rho+1}}{\rho (\rho+1)} \bigg| + \frac{1}{2} \bigg| \int_2^x \log( 1- t^{-2} ) dt\bigg|.$$
Sending $x$ to infinity, the integral can be evaluated to yield $\log(16/27)$, and the sum over the zeroes can be estimated on the Riemann hypothesis by
\begin{eqnarray*}
\bigg| \sum_{\rho} \frac{2^{\rho+1}}{\rho (\rho+1)} \bigg| & < & 2^{3/2}  \sum_{\rho} \frac{1}{| \rho|^2}.
\end{eqnarray*}
The value of this sum is explicitly known. By comparing (10) and (11) of Davenport \cite[Ch.\ 12]{davenport}, we have that
$$-2 \sum_{\gamma > 0} \frac{\beta}{| \rho|^2} = -\frac{\gamma}{2} -1 + \frac{1}{2} \log 4 \pi$$
where $\gamma$ is Euler's constant. Thus,
$$\sum_{\rho} \frac{1}{| \rho|^2} = \gamma - 2 + \log 4 \pi,$$
and the result follows.
\end{proof}

We now consider the existence of prime numbers in an interval of the form $(x-h, x+h)$. We do this by defining the weight function
\begin{displaymath}
   w(n) = \left\{
     \begin{array}{ll}
      1 - |n - x|/h & : \hspace{0.1in}  x-h < n < x+h\\
       0   & : \hspace{0.1in} \text{otherwise,}
     \end{array}
   \right.
\end{displaymath} 
and considering the identity
\begin{eqnarray}
\sum_{x-h<n<x+h} \Lambda(n) w(n) & = & \frac{1}{h} \Big(\psi_1(x+h) - 2 \psi_1(x) + \psi_1 (x-h)\Big).
\end{eqnarray}
One can verify this by expanding the weighted sum on the left hand side. An application of Lemma \ref{first} to the right side of the above gives us the following result.

\begin{lemma} \label{dog}
Let $x>0$ and $h>0$. Then
$$\sum_{x-h<n<x+h} \Lambda(n) w(n) = h - \frac{1}{h} \Sigma + \epsilon(h)$$
where 
$$\Sigma = \sum_{\rho} \frac{ (x+h)^{\rho+1} - 2x^{\rho+1} +(x-h)^{\rho+1}}{\rho(\rho+1)}$$
and 
$$|\epsilon(h)| < \frac{48}{5 h}.$$
\end{lemma}

We use this lemma to prove our results. Our focus is estimating the sum $\Sigma$, which we consider in two parts: $\Sigma = \Sigma_1 + \Sigma_2.$ Here, $\Sigma_1$ ranges over the zeroes $\rho$ with $|\gamma| < \alpha x / h$, where $\alpha>0$ is to be chosen later, and $\Sigma_2$ is the contribution from the remaining zeroes.

For $\Sigma_1$, we notice that the summand may be written as
$$\int_{x-h}^{x+h} (h-|x-u|) u^{\rho-1} du,$$
the absolute value of which can be bounded above by
$$\frac{1}{\sqrt{x-h}} \int_{x-h}^{x+h} (h-|x-u|) du = \frac{h^2}{\sqrt{x-h}} .$$
It follows that
\begin{eqnarray*}
\Sigma_1 & \leq & \frac{h^2}{\sqrt{x-h}} \sum_{ |\gamma| < \alpha x / h}  1 \\
& = & \frac{2h^2}{\sqrt{x-h}} N(\alpha x/h),
\end{eqnarray*}
where $N(T)$ is as before, and thus the bound (\ref{trudgiancount}) gives us that
\begin{equation} \label{sigma1}
| \Sigma_1 | < \frac{\alpha x h}{\pi \sqrt{x-h}} \log(\alpha x/h)
\end{equation}
when $\alpha x/h > 15$. We can estimate $\Sigma_2$ trivially on the Riemann hypothesis by
\begin{eqnarray*}
|\Sigma_2| & < & 4 (x+h)^{3/2} \sum_{|\gamma| > \alpha x/h} \frac{1}{\gamma^2} \\
& = & 8 (x+h)^{3/2} \sum_{\gamma > \alpha x/h} \frac{1}{\gamma^2} \\
& < & \frac{4 h (x+h)^{3/2}}{\pi \alpha x} \log(\alpha x/h),
\end{eqnarray*}
where the last line follows from Lemma 1 (ii) of Skewes \cite{skewes}. We also note that Lehman \cite{lehman} has explicit bounds for the sum $\sum_{\gamma > 0} \gamma^{-k}$, which could be useful when working with explicit formulas for the general sum (see Ingham \cite[Ch.\ 4]{ingham})
$$\psi_k(x) = \sum_{n \leq x} (x-n)^k \Lambda(n).$$ 
Now, putting our estimates for $\Sigma_1$ and $\Sigma_2$ into Lemma \ref{dog} we have
\begin{eqnarray*}
\sum_{n} \Lambda(n) w(n) & > & h - \frac{1}{h} ( |\Sigma_1| + |\Sigma_2|) - \frac{48}{5h} \\
& = & h - \bigg(\frac{\alpha x}{\pi \sqrt{x-h}} +\frac{4 (x+h)^{3/2}}{\pi \alpha x} \bigg) \log(\alpha x/h) - \frac{48}{5h}.
\end{eqnarray*}
Notice that as we will choose $h$ to be $o(x)$, it follows that the term in front of the logarithm is asymptotic to
$$ \Big(\frac{\alpha}{\pi} + \frac{4}{\pi \alpha} \Big) \sqrt{x}.$$
It is a straightforward exercise in differential calculus to show that $\alpha = 2$ will minimise this term, and thus
\begin{eqnarray*}
\sum_{n} \Lambda(n) w(n) & > &  h - \frac{ 2}{\pi} \bigg(\frac{ x}{ \sqrt{x-h}} +\frac{ (x+h)^{3/2}}{ x} \bigg) \log(2 x/h) - \frac{48}{5h},
\end{eqnarray*}
or rather
\begin{eqnarray*} 
\psi(x+h) - \psi(x-h) & = & \sum_{x-h < n \leq x+h} \Lambda(n) \\
& > &  h - \frac{ 2}{\pi} \bigg(\frac{ x}{ \sqrt{x-h}} +\frac{ (x+h)^{3/2}}{ x} \bigg) \log(2 x/h) - \frac{48}{5h}.
\end{eqnarray*}
As before, an application of Lemma \ref{psitheta} allows us to remove the contribution from the higher prime powers \textit{viz.}
\begin{eqnarray*}
\sum_{x - h < p \leq x + h} \log p & > &  h - \frac{ 2}{\pi} \bigg(\frac{ x}{ \sqrt{x-h}} +\frac{ (x+h)^{3/2}}{ x} \bigg) \log(2 x/h) \\
& & -(1+7.5 \cdot 10^{-7}) \sqrt{x+h} - 3 (x+h)^{1/3} +0.9999\sqrt{x-h} - \frac{48}{5h} 
\end{eqnarray*}
for all $x \geq 121$ and $\alpha x/h > 15$. If we set $h = d \sqrt{x} \log x$, the leading term on the right hand side can be shown to be asymptotic to
\begin{equation} \label{asym}
\Big( d - \frac{2}{\pi} \Big) \sqrt{x} \log x + \frac{4}{\pi} \sqrt{x} \log \log x.
\end{equation}
Thus, for $d \geq 2/\pi$ we have that there is a prime in the interval
$$(x-d \sqrt{x}\log x, x+ d \sqrt{x} \log x],$$
and so we choose $d = 2/ \pi$. Then, using a monotonicity argument we have this for all $x \geq 65000$. Replacing $x + d \sqrt{x} \log x$ with $x$, we have that there is a prime in the interval
$$(x - 2 d \sqrt{x} \log x, x]$$
for all 
$$x \geq 65000 + \frac{2}{\pi} \sqrt{65000} \log(65000) > 66798$$
where $2d = 4/\pi$. We complete the proof of Theorem \ref{one} by using \textsc{Mathematica} to verify the theorem for the remaining values of $x$.

It now remains to prove Theorem \ref{implicit} by showing that the constant $4/\pi$ can be reduced to $(1+\epsilon)$ by a more detailed analysis of the sum $\Sigma_1$. Bounding trivially, we have that
$$|\Sigma_1| \leq x^{3/2} \sum_{|\gamma| < \alpha x/h} \frac{|(1+h/x)^{3/2} e^{i \gamma \log(1+h/x)} + (1-h/x)^{3/2} e^{i \gamma \log(1-h/x)} - 2|}{\gamma^2}.$$
By noting the straightforward bound
$$\log(1 \pm h/x) = \pm\frac{h}{x} + O\bigg(\frac{h^2}{x^2} \bigg)$$
which holds for $h = o(x)$, we have that
\begin{eqnarray*}
e^{i \gamma \log (1 \pm h/x)} & = & e^{\pm i \gamma h/x} \Big( 1 + O \Big( \gamma \frac{h^2}{x^2} \Big) \Big) \\
& = & e^{\pm i \gamma h/x} + O(\alpha h/x)
\end{eqnarray*}
as $|\gamma| < \alpha x/h$. Using this estimate and 
$$(1 \pm h/x)^{3/2} = 1 + O(h/x),$$
one obtains 
\begin{equation*}
|\Sigma_1| \leq 8 x^{3/2} \sum_{0<\gamma < \alpha x/h} \frac{ \sin^2 (\frac{h \gamma}{2 x}) }{\gamma^2} + O(\alpha  h \sqrt{x}).
\end{equation*}
This sum can be estimated using Theorem A in Ingham's book \cite[Pg.\ 18]{inghambook} and (\ref{trudgiancount}) to get that
$$|\Sigma_1| \leq \frac{4 x^{3/2}}{\pi}  \int_{\gamma_1}^{\alpha x/h} \frac{ \log(u) \sin^2 (\frac{hu}{2x})}{u^2}du + O(\alpha  h \sqrt{x})$$
where $\gamma_1 = 14.1347\ldots$ denotes the height of the first nontrivial zero of the Riemann zeta-function. Employing the substitution $u =2xt/h$ and simplifying shows that
\begin{equation} \label{bekay}
| \Sigma_1 | \leq \bigg( \frac{2 }{\pi} \int_0^{\alpha/2} \frac{\sin^2 t}{t^2} dt \bigg) h \sqrt{x} \log(x/h)  + O(\alpha  h \sqrt{x}).
\end{equation}

Now, estimating $\Sigma_2$ as in the previous section, we have from Lemma \ref{dog} and (\ref{bekay}) that
$$\sum_{x-h < n < x+h} \Lambda(n) \geq h - \Big( \frac{4}{\pi \alpha} + \frac{2}{\pi} \int_0^{\alpha/2} \frac{\sin^2 t}{t^2} dt\Big) \sqrt{x} \log(x/h)+O(\alpha \sqrt{x}).$$
If we set $h = c \sqrt{x} \log x$, and choose
\begin{equation} \label{integralnote}
c >  \frac{2}{\pi \alpha} + \frac{1}{\pi} \int_0^{\alpha/2} \frac{\sin^2 t}{t^2} dt,
\end{equation}
then it follows that
$$\sum_{x-h < n < x+h} \Lambda(n) \gg \sqrt{x} \log x.$$
Note that the integral in (\ref{integralnote}) is equal to $\pi/2$ if one sends $\alpha$ to infinity. Therefore, we have $c = 1/2 + \epsilon$ provided that we take $\alpha$ to be sufficiently large. Of course, to avoid inflating the error term $O(\alpha \sqrt{x})$, one finds that $\alpha = \log \log x$ works fine. One can also remove the contribution of prime powers to the sum to show that there is a prime in the interval
$$(x-(1/2+\epsilon) \sqrt{x} \log x, x+(1/2+\epsilon) \sqrt{x} \log x]$$
for all sufficiently large values of $x$. This completes the proof of Theorem \ref{implicit}. 

Finally, as mentioned in Section 1, we can also show that Theorem \ref{cramer} can be taken with $c=3+\epsilon$ provided that $x$ is sufficiently large. For if we take
$$c = 1+  \frac{2}{\pi \alpha} + \frac{1}{\pi} \int_0^{\alpha/2} \frac{\sin^2 t}{t^2} dt = \frac{3}{2} + \epsilon,$$
then again removing the contribution from prime powers, we have that
$$\sum_{x-h < p < x+h} \log p \geq \sqrt{x} \log x + O(\sqrt{x} \log \log x).$$
The result then follows from using partial summation.

In the next chapter, we migrate from the problem of primes in short intervals to that of additive problems in number theory. Our main characters -- the primes and the zeroes of $\zeta(s)$  -- remain central to the story.

\newpage


\chapter{Some Results in Additive Number Theory}
\label{chapter4}

\begin{quote}
\textit{``Reeling and Writhing, of course, to begin with,'' the Mock Turtle replied; ``and then the different branches of Arithmetic -- Ambition, Distraction, Uglification, and Derision.''}
\end{quote}

Additive number theory deals with the additive properties of numbers. Specifically, we seek analogues to the fundamental theorem of arithmetic where the operation is addition instead of multiplication. As mentioned in Chapter 1, the following two conjectures of Goldbach form the centrepiece of this theory.

\begin{conj}[The Binary Goldbach Conjecture]
Every even integer greater than two can be written as the sum of two primes.
\end{conj}

\begin{conj}[The Ternary Goldbach Conjecture]
Every odd integer greater than five can be written as the sum of three primes.
\end{conj}

The latter of these conjectures was recently proven by Helfgott \cite{helfgott2013}. Vinogradov \cite{vinogradov} first proved in 1937 that every odd integer $n > C$ can be written as the sum of three primes (provided that $C$ is sufficiently large). Explicit results for $C$ appeared soon after; Borodzin showed in 1939 that the one could take $C=3^{3^{15}}$. This was improved to $C=3.33 \cdot 10^{43000}$ by Chen and Wang \cite{chenwang} and to $C=2 \cdot 10^{1346}$ by Liu and Wang \cite{liuwanggoldbach}. Helfgott finished this line of work, showing that $C=5$ does the trick.

The binary conjecture, on the other hand, has not submitted itself as gently as the ternary conjecture. We do not yet even know if every sufficiently large even integer may be written as the sum of two primes. However, by a result of Chen~\cite{chen}, we know that every sufficiently large even integer can be written in the form $p+r$, where $p$ is a prime and $r$ is either a prime or a product of two primes. Another impressive step in the right direction is the result of Linnik, that every sufficiently large integer $n$ can be written as
$$n=p_1+p_2+2^{a_1}+2^{a_2}+\cdots + 2^{a_r}$$
where $p_1$ and $p_2$ are primes, $a_i$ is a positive integer for all $i$, and $r \leq K$ where $K$ is an absolute constant. One can see the paper of Heath-Brown and Puchta \cite{heathbrownpuchta} for a short history on the explicit bounds for $K$. Both of these results are weak forms of Goldbach's binary conjecture, and it would be interesting to see if one could make either of these explicit in the same style as Helfgott's theorem. That is, can we furnish an actual number $C$ such that all even integers $n > C$ can be written in at least one of the above two ways? Of course, as both of these are weak versions of Goldbach's binary conjecture, we would optimistically expect $C=2$ to work.

In this chapter, we prove two explicit results in the additive theory of numbers. Our first result is a weak form of Goldbach's binary conjecture; we prove that every integer greater than two can be written as the sum of a prime and a square-free number. This completes a theorem of Estermann \cite{estermann}, who proved the same theorem but only for sufficiently large integers. 

We then examine the result of Erd\H{o}s that every sufficiently large integer $n$ such that $n \not\equiv 1 \mod 4$ can be written as the sum of the square of a prime and a square-free number. We complete this result, showing that this is true for all $n \geq 10$ satisfying this congruence condition. This is joint work with Dave Platt at the University of Bristol, and so I will be clear in stating our individual contributions.

Importantly, the two results presented in this chapter serve to measure the current state of play in the explicit theory of numbers. We exist in a time where long-standing results may be revisited and sharpened, and as such we provide a bounty of open problems for the interested reader in Chapter 6.

\section{On a theorem of Estermann}

It was first shown by Estermann \cite{estermann} in 1931 that every sufficiently large positive integer $n$ can be written as the sum of a prime and a square-free number, that is, 
$$n = p_1 + p_2 p_3 \cdots p_m$$
where $p_i$ is prime for all $1 \leq i \leq m$ and $p_i \neq p_j$ for $2 \leq i < j \leq m$. There is no restriction on the size of $m$ here, however we already know by the later work of Chen \cite{chen} that one can take $m \leq 3$.

Moreover, Estermann proved that the number $T(n)$ of such representations satisfies the asymptotic formula
\begin{equation} \label{estermannformula}
T(n) \sim \frac{c n}{\log n} \prod_{ p | n} \bigg( 1+\frac{1}{p^2-p-1} \bigg),
\end{equation}
where 
\begin{equation} \label{artin}
c=\prod_p \bigg( 1 - \frac{1}{p(p-1)} \bigg) = 0.3739558\ldots
\end{equation}
is a product over all prime numbers which is known as Artin's constant (see Wrench's computation \cite{wrench} for more details). Notably, it follows that $T(n) \gg \pi(n)$, and so we expect $(n-p)$ to be square-free for a positive proportion of primes $p$ below $n$.

In 1935, Page \cite{page} improved Estermann's result by giving a bound for the order of the error term in (\ref{estermannformula}), using estimates for the error in the Prime Number Theorem for arithmetic progressions. Mirsky \cite{mirsky} extended these results in 1949 to count representations of an integer as the sum of a prime and a $k$-free number, that is, a number which is not divisible by the $k$-th power of any prime. More recently, in 2005, Languasco \cite{languasco} treated the possibility of Siegel zeroes (see Davenport \cite{davenport} for a discussion on this) with more caution so as to provide better bounds on the error.

The purpose of this section is to prove the following theorem, completing the result of Estermann. This work has been accepted for publication \cite{dudekestermann} in \textit{The Ramanujan Journal}.

\begin{theorem} \label{estermann}
Every integer greater than two is the sum of a prime and a square-free number. 
\end{theorem}

We prove this theorem by working in the same manner as Mirsky \cite{mirsky}, though we employ explicit estimates on the error term for the Prime Number Theorem in arithmetic progressions. Specifically, if we let
$$\theta(x;k,l) = \sum_{\substack{p \leq x \\ p \equiv l (\text{mod } k)}} \log p$$
where the sum is over primes $p$, we require estimates of the form
\begin{equation} \label{error}
\bigg| \theta(x;k,l) - \frac{x}{\varphi(k)} \bigg| < \frac{\epsilon x}{\varphi(k)}
\end{equation}
where $\epsilon>0$ is sufficiently small and $x$ and $q$ are suitably ranged. Good estimates of this type are available due to Ramar\'{e} and Rumely \cite{ramarerumely}, but are only provided for finitely many moduli $k$. It turns out, however, that this is sufficient, as the Brun--Titchmarsh theorem (which we will state soon) is enough for the remaining cases. Later in the chapter -- when we work on the result of Erd\H{o}s -- Platt actually improves the Ramar\'{e}--Rumely estimates for the Prime Number Theorem in arithmetic progressions. We delay this improvement until then, as it is not required in the proof of Theorem \ref{estermann}.

As with many of the proofs in this thesis, we will first employ analytic methods to obtain explicit bounds. In this section, this will allow us to prove the theorem in the range $n \geq 10^{10}$. Then we verify the remaining cases computationally. 

As such, we let $n\geq10^{10}$ be a positive integer and $\mu: \mathbb{N} \rightarrow \{-1,0,1\}$ denote the M\"obius function, where $\mu(n)$ is zero if $n$ is not square-free; otherwise $\mu(n)=(-1)^{\omega(n)}$ where $\omega(n)$ denotes the number of distinct prime factors of $n$. We also stipulate that $\mu(1)=1$. For a positive integer $n$, it can be shown that the sum
$$\mu_2 (n) = \sum_{a^2 | n} \mu(a)$$
is equal to $1$ if $n$ is square-free and zero otherwise. Thus, it follows that the expression
$$T(n)=\sum_{p \leq n} \mu_2(n-p)$$
counts the number of ways that $n$ may be expressed as the sum of a prime and a square-free number. We will employ logarithmic weights so as to use the known prime number estimates with more ease, and so we define
$$R(n) = \sum_{p \leq n} \mu_2(n-p) \log p.$$
We note that $n$ is the sum of a prime and a square-free number if and only if $R(n) >0$. As such, the majority of this paper is dedicated to finding a lower bound for $R(n)$. The expression for $R(n)$ can be rearranged so as to involve weighted sums over the prime numbers in arithmetic progressions:
\begin{eqnarray*}
R(n) & = & \sum_{p \leq n} \log p \sum_{a^2 | (n-p)} \mu(a) \\
& = & \sum_{a \leq n^{1/2}} \mu(a) \sum_{\substack{p \leq n \\ a^2 | (n-p)}} \log p \\
& = & \sum_{a \leq n^{1/2}} \mu(a) \theta(n; a^2, n).
\end{eqnarray*}
We will split the range of this sum into three parts, for we shall use a different technique to bound each of them. Thus, we may write
\begin{equation} \label{bigdaddy}
R(n) = \Sigma_1 + \Sigma_2 + \Sigma_3 
\end{equation}
where
\begin{eqnarray*}
\Sigma_1 & = & \sum_{a \leq 13 } \mu(a) \theta(n; a^2, n), \\
\Sigma_2 & = & \sum_{13 < a \leq n^A } \mu(a) \theta(n; a^2, n), \\
\Sigma_3 & = & \sum_{n^A < a \leq n^{1/2} } \mu(a) \theta(n; a^2, n),
\end{eqnarray*}
and $A \in (0,1/2)$ is to be chosen later to minimise the value of $N$ such that $R(n)>0$ for all $n>N$. We will use the estimates of Ramar\'{e} and Rumely \cite{ramarerumely} to bound $\Sigma_1$; this is the reason for the specific range of $a$ in this sum. We will then use the Brun--Titchmarsh theorem to bound $\Sigma_2$. Finally, $\Sigma_3$ will be bounded using trivial estimates. 

Starting with the first sum, Theorem 1 of Ramar\'{e} and Rumely \cite{ramarerumely} provides estimates of the form
\begin{equation}
\bigg| \theta(n;a^2,n) - \frac{n}{\varphi(a^2)} \bigg| < \frac{\epsilon_a n}{\varphi(a^2)}.
\end{equation}
In particular, by looking through the square moduli in Table 1 of their paper, we have values of $\epsilon_a$ for all $1 \leq a \leq 13$ which are valid for all $n\geq 10^{10}$. We therefore have that
\begin{eqnarray} \label{big1}
\Sigma_1 & > & n \sum_{ a \leq 13} \bigg( \frac{\mu(a)}{\varphi(a^2)} - \epsilon_a \frac{\mu^2(a)}{\varphi(a^2)} \bigg) \nonumber \\
& > &  n \bigg( \sum_{a=1}^{\infty} \frac{\mu(a)}{\varphi(a^2)} - \sum_{a>13 } \frac{\mu(a)}{\varphi(a^2)}- \sum_{a \leq 13} \frac{ \epsilon_a \mu^2(a)}{\varphi(a^2)} \bigg).
\end{eqnarray}
We now wish to estimate the three sums in the above parentheses. We note that the leftmost sum is equal to Artin's constant (\ref{artin}) \textit{viz.}
\begin{eqnarray*}
\sum_{a=1}^{\infty} \frac{\mu(a)}{\varphi(a^2)} = \prod_{p} \bigg(1- \frac{1}{p(p-1)}\bigg)=c.
\end{eqnarray*}
Wrench \cite{wrench} has computed this constant to high accuracy; it will suffice for the purpose of Theorem \ref{estermann} to note that $c>0.373$. 

We will, for the moment, neglect the middle sum in (\ref{big1}), for it shall be considered jointly with a term in the estimation of $\Sigma_2$. Thus, in our estimation of $\Sigma_1$, it remains to manually compute the upper bound for the rightmost sum. This is a straightforward task which is done in reference to Table 1 of Ramare and Rumely's paper \cite{ramarerumely}. We get that
$$ \sum_{a \leq 13} \frac{\epsilon_a \mu^2(a)}{\varphi(a^2)} < 0.005.$$

We now bring $\Sigma_2$ into the fray by calling on Montgomery and Vaughan's \cite{MV} explicit version of the Brun--Titchmarsh theorem:
\begin{theorem} \label{bt}
Let $x>0$ and $y>0$ be real numbers, and let $k$ and $l$ be relatively prime positive integers with $k<y$. Then
$$\pi(x+y;k,l) - \pi(x;k,l) < \frac{2 y}{\varphi(k) \log(y/k)}$$
where $\pi(x;k,l)$ denotes the number of prime numbers $p \leq x$ such that $p \equiv l \mod k$.
\end{theorem}
A trivial application of Theorem \ref{bt} effects the bound
$$\theta(n; a^2,n) < 2 \bigg(\frac{\log n}{\log(n/a^2)} \bigg) \frac{n}{\varphi(a^2)}.$$
Clearly, in the range $13 < a \leq n^{A}$ we may bound 
$$\frac{\log n}{\log(n/a^2)} \leq \frac{1}{1-2A}$$
and so we have the estimate that
$$\theta(n; a^2,n) = \frac{n}{\varphi(a^2)} + \epsilon \bigg(\frac{1+2A}{1-2A} \bigg)\frac{n}{\varphi(a^2)}$$
where $|\epsilon| < 1$. We may then bound $\Sigma_2$ from below by
$$\Sigma_2 > n \sum_{13 < a \leq n^A} \frac{\mu(a)}{\varphi(a^2)} - n \bigg(\frac{1+2A}{1-2A}\bigg) \sum_{13 < a \leq n^A } \frac{ \mu^2(a)}{\varphi(a^2)}.  $$
We can then add this to our estimate for $\Sigma_1$ to get
\begin{eqnarray} \label{sigma1and2}
\Sigma_1 + \Sigma_2 & > & n \bigg( c - 0.005 - \sum_{a> n^A} \frac{\mu(a)}{\varphi(a^2)} - \bigg(\frac{1+2A}{1-2A}\bigg) \sum_{13 < a \leq n^A} \frac{ \mu^2(a)}{\varphi(a^2)} \bigg) \nonumber \\
& > & n \bigg(c -0.005 - \bigg(\frac{1+2 A}{1- 2 A} \bigg) \sum_{a > 13} \frac{\mu^2 (a)}{\varphi(a^2)}\bigg).
\end{eqnarray}
We estimate the sum in the above inequality by writing it as follows:
\begin{eqnarray*}
\sum_{a > 13} \frac{\mu^2 (a)}{\varphi(a^2)} =\sum_{a=1}^{\infty} \frac{\mu^2(a)}{\varphi(a^2)} - \sum_{a\leq 13} \frac{\mu^2(a)}{\varphi(a^2)}.
\end{eqnarray*}
The infinite sum is less than 1.95 (see Ramar\'{e} \cite{ramare} for example), and the finite sum can be computed by hand to see that the sum in (\ref{sigma1and2}) is bounded above by 0.086. Thus
\begin{equation} \label{sigma12}
\Sigma_1 + \Sigma_2 > n \bigg(c - 0.005 - \bigg(\frac{1+2A}{1-2A}\bigg)(0.086)\bigg).
\end{equation}

For $\Sigma_3$, we have trivially that
\begin{eqnarray} \label{sigma3bound}
|\Sigma_3| & < & \sum_{n^{A} < a \leq n^{1/2}} \theta(n;a^2,n) \nonumber  \\
& < & \sum_{n^{A} < a \leq n^{1/2}} \bigg(1+\frac{n}{a^2} \bigg) \log n. \nonumber 
\end{eqnarray}
This follows from the fact that there can be at most $([n/a^2]+1)$ integers not exceeding $n$ that are divisible by $a^2$. Therefore, we have that
\begin{eqnarray}
|\Sigma_3| & < & n^{1/2} \log n + n \log n \sum_{n^{A} < a \leq n^{1/2}} \frac{1}{a^2} \nonumber \\
& < & n^{1/2} \log n + n^{1-2A} \log n + n \log n \int_{n^{A}}^{n^{1/2}} t^{-2} dt \nonumber \\
& = & n^{1-2A} \log n + n^{1-A} \log n.
\end{eqnarray}

We can now provide an explicit lower bound for $R(n)$. Combining our explicit estimates (\ref{sigma12}) and (\ref{sigma3bound}) with (\ref{bigdaddy}) and dividing through by $n$ we have that
\begin{eqnarray*}
\frac{R(n)}{n} & > & c- 0.005-\bigg(\frac{1+2A}{1-2A}\bigg)(0.086) \\
& - &n^{-1/2} \log n - n^{-2A} \log n - n^{-A} \log n.
\end{eqnarray*}
For sufficiently small $A$ and large $n$, the right hand side will be positive. For any $n$, we have $c > 0.373$; it is a simple matter to choose $A=1/4$ and verify that the right hand side is positive for all $n \geq 10^{10}$. That is, Theorem \ref{estermann} is true for all integers $n \geq 10^{10}$.

It so remains to prove this result for all integers in the range $3 \leq n <10^{10}$. If $n$ is even, we have the numerical verification by Oliviera e Silva, Herzog and Pardi \cite{silva} that all even integers up to $4 \cdot 10^{18}$ can be written as the sum of two primes. Thus, every even integer greater than two may be written as the sum of a prime and a square-free number. 

Therefore, we need to check that every odd integer $3 \leq n < 10^{10}$ can be written as the sum of a prime and a square-free number. Our algorithm is as follows: we partition the range $3 \leq n < 10^{10}$ into intervals of size $10^7$. That is, we are considering the intervals $I_a=(a \cdot 10^7, (a+1) \cdot 10^7)$ for all integers $a$ such that $0 \leq a < 1000$. 

The interval $I_0$ is straightforward to check; in general, for the interval $I_a$, we use \textsc{Mathematica} to generate a decreasing list $(p_1,p_2,\ldots, p_{100})$ of the 100 greatest primes which do not exceed $a \cdot 10^7$. Then, starting with the first odd number $n \in I_a$, we check to see first if $n-p_i$ is square-free as $i$ ranges from 1 to 100, moving to $n+2$ as soon as we have found a prime which works. The purpose of choosing primes close to the interval is soon apparent, for  the relatively small value of $n-p_i$ makes it a simple exercise to check for square-freeness. This is a straightforward computation; we ran this on \textsc{Mathematica} and it took just under 3 days on a 2.6GHz laptop. This computation could no doubt be optimised; for example, we could have first generated a list of the square-free numbers up to some bound by employing the sieve of Eratosthenes but with square moduli. Then, checking that $n-p_i$ is an element of this list would be far quicker than running a square-free test for each number. 

\section{On a theorem of Erd\H{o}s}

In 1935, quite soon after the aforementioned theorem of Estermann \cite{estermann}, Erd\H{o}s \cite{erdos} proved that every sufficiently large integer $n \not\equiv 1 \text{ mod } 4$ may be written as the sum of the square of a prime and a square-free number.  The congruence condition here is sensible: if $n\equiv 1 \mod 4$ then $4|(n-p^2)$ for any odd prime $p$, and so $(n-p^2)$ is clearly not square-free. This only leaves the case $p=2$, but $n-4$ fails to be square-free infinitely often\footnote{For example, one can consider the congruence class $13 \text{ mod } 36$.}.

It is the objective of this section to make explicit the proof provided by Erd\H{o}s, to the end of proving the following theorem. 

\begin{theorem} \label{erdos}
Let $n \geq 10$ be an integer such that $n \not\equiv 1 \mod 4$. Then there exists a prime $p$ and a square-free number $m$ such that $n=p^2+m$.
\end{theorem}

The proof of this theorem, and thus the content of this section, was completed jointly with Dave Platt at the University of Bristol. This work has now been accepted by the \textit{LMS Journal of Computation and Mathematics} \cite{dudekplatt2}. 

The reader should note that Theorem \ref{erdos} is, in some sense, stronger than Theorem \ref{estermann}, for the sequence of squares of primes is far more sparse than the sequence of primes. We prove Theorem \ref{erdos} in a similar way, by combining modern explicit results on primes in arithmetic progressions and computation. We push both of these resources to their limits in doing so, at least, a lot further than in the last section.

Specifically, we extend the results of Ramar\'{e} and Rumely \cite{ramarerumely}, and our computational algorithm is far more involved than that of the previous section. These last two important points were significant contributions on Platt's part, and before these were brought to the table, I could only prove Theorem \ref{erdos} for all $n \geq 10^{17}$. Therefore, whilst the method of this section is my adaptation of Erd\H{o}s' original proof, it was Platt's contributions that allowed the proof to get over the line.

The proof may be roughly outlined as follows. For any integer $n$ satisfying the conditions of the above theorem, we want to show that there exists a prime $p < \sqrt{n}$ such that $n-p^2$ is square-free. That is, we require some prime $p$ such that
$$n-p^2 \not\equiv 0 \mod q^2$$
for all odd primes $q < \sqrt{n}$. The idea is to consider, for some large $n$ and each odd prime $q< \sqrt{n}$, those \textit{mischievous} primes $p$ that satisfy the congruence
$$n \equiv p^2 \mod q^2.$$
Then, for each $q$ we explicitly bound from above (with logarithmic weights) the number of such primes $p$. Summing over all moduli $q$ gives us an upper bound for the weighted count of the mischievous primes \textit{viz.}
$$\sum_{q < \sqrt{n}} \sum_{\substack{p < \sqrt{n} \\ n \equiv p^2 \text{ mod } q^2}} \log p.$$
It is then straightforward to show that for large enough $n$, the above sum is less than the weighted count of \emph{all} primes less than $\sqrt{n}$, and therefore there must exist a prime $p < \sqrt{n}$ such that $n-p^2$ is not divisible by the square of any prime. 

This method works well, and allows us first to prove Theorem \ref{erdos} for all integers $n \geq 2.5 \cdot 10^{14}$ which satisfy the congruence condition. We then eliminate the remaining cases by direct computation to complete the proof.

Note that from before, the paper of Ramar\'{e}--Rumely \cite{ramarerumely} provides us with bounds on the  error term in the Prime Number Theorem in arithmetic progressions. Specifically, these are of the form
\begin{equation*}
\bigg| \theta(x;k,l) - \frac{x}{\varphi(k)} \bigg| < \epsilon(k,x_0) \frac{x}{\varphi(k)}
\end{equation*}
and
\begin{equation*}
\bigg| \theta(x;k,l) - \frac{x}{\varphi(k)} \bigg| < \omega(k,x_1) \sqrt{x}
\end{equation*}
for various ranges of $x\geq x_0$ and $x\leq x_1$ respectively. These computations were in turn based on Rumely's numerical verification of the generalised Riemann hypothesis (GRH) \cite{Rumely1993} for various moduli and to certain heights. 

The GRH is a more general form of the Riemann hypothesis, and asserts that a broad class of Dirichlet series (namely the Dirichlet $L$-functions) have all of their nontrivial zeroes on the line $\text{Re}(s)=1/2$ (see Davenport \cite{davenport} for further discussion). Since the computations of Rumely, Platt has verified GRH for a wider range of moduli and to greater heights \cite{Platt2013}. For our purposes, we rely only on the following result.
\begin{lemma}\label{grh}
Let $q$ be a prime satisfying $17 \leq q \leq 97$. All nontrivial zeroes $\rho$ of Dirichlet \emph{L}-functions derived from characters of modulus $q^2$ with $\emph{Im}(\rho) \leq 1000$ have $\emph{Re}(\rho)=1/2$.
\end{lemma} 
\begin{proof}
See Theorem 10.1 of Platt's paper \cite{Platt2013}.
\end{proof}
We can therefore extend the results of Ramar\'{e}--Rumely with the following lemma.

\begin{lemma}
For $x>10^{10}$ we have 
\begin{equation*}
\bigg| \theta(x;q^2,l) - \frac{x}{\varphi(q^2)} \bigg| < \epsilon(q^2,10^{10}) \frac{x}{\varphi(q^2)}
\end{equation*}
for the values of $q$ and $\epsilon(q^2,10^{10})$ in Table \ref{tab:new_eq}.
\end{lemma}
\begin{proof}
We refer to Ramar\'{e} and Rumely \cite{ramarerumely}. The values for $q\in\{3,5,7,11,13\}$ are from Table 1 of that paper. For the other entries, we use Theorem 5.1.1 with $H_\chi=1000$ and $C_1(\chi,H_\chi)=9.14$ (see display 4.2). We set $m=10$ for $q\leq 23$, $m=12$ for $q\geq 47$ and $m=11$ otherwise. We use $\delta=2e/H_\chi$ and for $\widetilde{A}_\chi$ we use the upper bound of Lemma 4.2.1. Finally, for $\widetilde{E}_\chi$ we rely on Lemma 4.1.2 and we note that $2\cdot 9.645908801\cdot\log^2(1000/9.14)\geq \log 10^{10}$ as required.
\end{proof}

\begin{table}[h!] 
\caption{Values for $\epsilon(q^2,10^{10})$.} 
\label{tab:new_eq} 
\centering 
\begin{tabular}{| c c | c c | c c | c c |} 
 \hline 
$q$ & $\epsilon(q^2,10^{10})$ & $q$ & $\epsilon(q^2,10^{10})$ & $q$ & $\epsilon(q^2,10^{10})$ & $q$ & $\epsilon(q^2,10^{10})$  \\[0.5 ex] \hline
$3$ & $0.003228$ & $19$ & $0.17641$ & $43$ & $0.95757$ & $71$ & $2.82639$\\
$5$ & $0.012214$ & $23$ & $0.25779$ & $47$ & $1.15923$ & $73$ & $3.00162$\\
$7$ & $0.017015$ & $29$ & $0.41474$ & $53$ & $1.50179$ & $79$ & $3.56158$\\
$11$ & $0.031939$ & $31$ & $0.47695$ & $59$ & $1.89334$ & $83$ & $3.96363$\\
$13$ & $0.042497$ & $37$ & $0.69397$ & $61$ & $2.03488$ & $89$ & $4.61023$\\
$17$ & $0.14271$ & $41$ & $0.86446$ & $67$ & $2.49293$ & $97$ & $5.55434$\\
\hline 
\end{tabular} 
\end{table} 

\begin{lemma}
We have 
$$\omega\left(3^2,10^{10}\right)=1.109042,$$ 
$$\omega\left(5^2,10^{10}\right)=0.821891,$$ 
$$\omega\left(7^2,10^{10}\right)=0.744132,$$
$$\omega\left(11^2,10^{10}\right)=0.711433$$ 
and 
$$\omega\left(13^2,10^{10}\right)=0.718525.$$ 
If $q$ is a prime such that $17 \leq q \leq 97$ we have
\begin{equation*}
\omega(q^2,10^{10})=\frac{\log 7 - \frac{7}{\varphi(q^2)}}{\sqrt{7}}.
\end{equation*}
\end{lemma}
\begin{proof}
The results for $\{3^2,5^2,7^2,11^2,13^2\}$ are from Table 2 of \cite{ramarerumely} with a slight correction to the entry for $5^2$. A short computation shows that the maximum occurs for all of the other $q$ when $x=7$ and $a=7$.
\end{proof}

\begin{lemma} \label{apbounds}
Let $T=\sqrt{2.5\cdot 10^{14}}$. Then for $x\geq T$ and $q\leq 97$ an odd prime we have
$$\bigg| \theta(x;q^2,l) - \frac{x}{\varphi(q^2)} \bigg| < \epsilon\left(q^2,T\right) \frac{x}{\varphi(q^2)},$$
where the values of $\epsilon\left(q^2,T\right)$ are given in Table \ref{table1}.
\end{lemma}

\begin{proof}
Using $\omega(q^2,10^{10})$ we have
\begin{equation*}
\bigg| \theta(T;q^2,l) - \frac{T}{\varphi(q^2)} \bigg| < \omega\left(q^2,10^{10}\right)\sqrt{T}
\end{equation*}
so for $x\in[T,10^{10}]$ we have
\begin{equation*}
\bigg| \theta(x;q^2,l) - \frac{x}{\varphi(q^2)} \bigg| < \frac{\omega\left(q^2,10^{10}\right)\varphi(q^2)}{\sqrt{T}}\frac{x}{\varphi(q^2)}
\end{equation*}
and so we can take 
\begin{equation*}
\epsilon\left(q^2,T\right)=\max\left(\epsilon\left(q^2,10^{10}\right),\frac{\omega\left(q^2,10^{10}\right)\varphi(q^2)}{\sqrt{T}}\right).
\end{equation*}
\end{proof}

\begin{table}[h!] 
\caption{Values for $\epsilon(q^2,T)$ for Lemma \ref{apbounds}.} 
\label{table1} 
\centering 
\begin{tabular}{c c c c c c c c } 
 \hline\hline 
$q$ & $\epsilon(q^2,T)$ & $q$ & $\epsilon(q^2,T)$ & $q$ & $\epsilon(q^2,T)$ & $q$ & $\epsilon(q^2,T)$  \\[0.5 ex] \hline
$3$ & $0.00323$ & $19$ & $0.17641$ & $43$ & $0.95757$ & $71$ & $2.82639$\\
$5$ & $0.01222$ & $23$ & $0.25779$ & $47$ & $1.15923$ & $73$ & $3.00162$\\
$7$ & $0.01702$ & $29$ & $0.41474$ & $53$ & $1.50179$ & $79$ & $3.56158$\\
$11$ & $0.03194$ & $31$ & $0.47695$ & $59$ & $1.89334$ & $83$ & $3.96363$\\
$13$ & $0.04250$ & $37$ & $0.69397$ & $61$ & $2.03488$ & $89$ & $4.61023$\\
$17$ & $0.14271$ & $41$ & $0.86446$ & $67$ & $2.49293$ & $97$ & $5.55434$\\
\hline\hline 
\end{tabular} 
\end{table} 

The reader should note that the above sharpening of Ramar\'{e} and Rumely's results could be extended so as to include all moduli in the given range and not just the squares. This would then allow one to ease the computation involved in proving Theorem \ref{estermann} in the previous section.

Now, let $n \geq 2.5\cdot 10^{14}$ be such that $n \not\equiv 1 \text{ mod } 4$ and consider the case where $q$ is an odd prime such that $q \leq 97$. We want to bound from above the number of primes $p < \sqrt{n}$ satisfying
\begin{equation} \label{congruence}
n \equiv p^2 \text{ mod } q^2.
\end{equation}
Clearly, $p$ can belong to at most two arithmetic progressions moduluo $q^2$. Therefore, by Lemma \ref{apbounds},  we can estimate the weighted count of such primes as follows.
$$\sum_{\substack{p < \sqrt{n} \\ n \equiv p^2 \text{ mod } q^2}} \log p \leq  \theta(\sqrt{n}; q^2, l) + \theta(\sqrt{n}; q^2, l')  < \frac{2 (1+\epsilon(q^2,T))}{q (q-1)} \sqrt{n}$$
where $l$ and $l'$ are the possible congruence classes for $p$ and $\epsilon(q^2,T)$ is given in Table \ref{table1}. Summing this over all $24$ values of $q$ gives us the contribution
\begin{equation} \label{cont1}
\sum_{q \in \{3,\ldots,97\}} \sum_{\substack{p < \sqrt{n} \\ n \equiv p^2 \text{ mod } q^2}} \log p <0.568 \sqrt{n}.
\end{equation}

We now consider the case where $97 < q \leq n^{c}$ and $c \in (0,1/4)$ is to be chosen later to achieve an optimal result. As in the previous section, the Brun--Titchmarsh Theorem gives us that
$$\pi(x;q,a) \leq \frac{2 x}{\varphi(q) \log(x/q)}$$
for all $x>q$. Trivially, one has that
$$\theta(\sqrt{n}; q^2,a) \leq \frac{\sqrt{n}}{q(q-1)} \frac{ \log n}{\log(\sqrt{n}/q^2)}.$$
As $q < n^{c}$, it follows that
\begin{equation} \label{cont1point5}
\sum_{97 < q \leq n^{c}} \sum_{\substack{p < \sqrt{n} \\ n \equiv p^2 \text{ mod } q^2}} \log p < \frac{\sqrt{n}}{\frac{1}{4}- c}  \sum_{97 < q \leq n^{c}} \frac{1}{q(q-1)}.
\end{equation}
We bound the sum as follows:
\begin{eqnarray*}
\sum_{97 < q \leq n^{c}} \frac{1}{q(q-1)} & < & \sum_{97 < q < 1000001} \frac{1}{q (q-1)} + \sum_{ n \geq 1000001} \frac{1}{n(n-1)} \\
& = & \sum_{97 < q <1000001} \frac{1}{q (q-1)} + \frac{1}{1000000} <0.00183.
\end{eqnarray*}
Substituting this into (\ref{cont1point5}) gives us that
\begin{equation} \label{cont2}
\sum_{97 < q \leq n^{c}} \sum_{\substack{p < \sqrt{n} \\ n \equiv p^2 \text{ mod } q^2}} \log p < \frac{0.00183 \sqrt{n}}{\frac{1}{4}-c}.
\end{equation}

Now, let $q$ be an odd prime such that $n^{c} < q < A \sqrt{n}$, where $A \in (0,1)$ is to be chosen later for optimisation. Since there are at most two possible residue classes modulo $q^2$ for $p$, the number of primes $p$ such that $n \equiv p^2 \text{ mod } q^2$ is trivially less than
$$2\bigg( \frac{\sqrt{n}}{q^2} + 1\bigg).$$
Clearly, using our logarithmic weights one has that
\begin{equation*}
\sum_{\substack{p < \sqrt{n} \\ n \equiv p^2 \text{ mod } q^2}} \log p <  \bigg( \frac{\sqrt{n}}{q^2} + 1\bigg) \log n
\end{equation*}
and so
\begin{eqnarray*}
\sum_{n^{c} < q < A \sqrt{n}} \sum_{\substack{p < \sqrt{n} \\ n \equiv p^2 \text{ mod } q^2}} \log p < \sqrt{n} \log n \sum_{m > n^{c}} \frac{1}{m^2} + \pi(A \sqrt{n}) \log(n)
\end{eqnarray*}
where $\pi(x)$ denotes the number of primes not exceeding $x$. The sum can be estimated in a straightforward way by 
$$\sum_{m > n^{c}} \frac{1}{m^2} < \frac{1}{n^{2c}} + \int_{n^{c}}^{\infty} \frac{1}{t^2} dt = \frac{1}{n^{2c}}+\frac{1}{n^c}$$
and by Theorem 6.9 of Dusart \cite{dusart}, which gives us that
$$\pi(A \sqrt{n}) < \frac{A \sqrt{n}}{\log(A \sqrt{n})}\bigg(1+\frac{1.2762}{\log (A \sqrt{n})}\bigg).$$
Note that, for the eventual choice of $A$ and our range of $n$, we will have $A \sqrt{n}>1$ and so this upper bound is valid. Therefore, putting this all together we have
\begin{equation} \label{cont3}
\sum_{n^c<q<A \sqrt{n}} \sum_{\substack{p < \sqrt{n} \\ n \equiv p^2 \text{ mod } q^2}} \log p < \sqrt{n}(n^{-2c}+n^{-c})\log n + \frac{A \sqrt{n} \log n}{\log(A \sqrt{n})}\bigg(1+\frac{1.2762}{\log (A \sqrt{n})}\bigg).
\end{equation}

Finally, we consider the range $A \sqrt{n} \leq q < \sqrt{n}$. If $n-p^2$ is divisible by $q^2$, then
\begin{equation} \label{blah}
n = p^2 + B q^2
\end{equation}
for some positive integer $B < A^{-2}$. We will need some preliminary results here. First, it is known by the theory of quadratic forms (see Davenport \cite[Ch. 6]{davenport}) that the equation
$$ax^2+by^2=n,$$
where $a, b$ and $n$ are given positive integers, has at most $s 2^{\omega(n)}$ proper solutions, that is, solutions with $\gcd(x,y)=1$. Note that $s$ denotes the number of automorphs of the above form and $\omega(n)$ denotes the number of different prime factors of $n$. The number of automorphs is directly related to the discriminant of the form; specifically, $s=4$ for the case $B=1$ and $s=2$ for $B>1$. Moreover, we are only interested in the case where $x$ and $y$ are both positive, and so it follows that (\ref{blah}) has at most $s 2^{\omega(n)-2}$ proper solutions. Finally, noting that there will be at most 1 improper solution to (\ref{blah}), namely $p=q$, we can bound the overall number of solutions to (\ref{blah}) by $s 2^{\omega(n)-2}+1$.

Furthermore, Theorem 11 of Robin \cite{robin}  gives us the explicit bound
\begin{equation} \label{robinbound}
\omega(n) \leq 1.3841 \frac{\log n}{\log \log n} 
\end{equation}
for all $n \geq 3$. Robin also gives (Theorem 12) the stronger bound
$$\omega(n) \leq  \frac{\log n}{\log \log n} + 1.45743 \frac{\log n}{(\log \log n)^2} $$
for all $n \geq 3$, which allows one to replace the value $1.3841$ in (\ref{robinbound}) with 1.1189 for all $n \geq 2.5 \cdot 10^{14}$. It turns out that $1.3841$ will be good enough for our proof, however one should note that using the constant $1.1189$ would ease the computation slightly by allowing us to prove Theorem \ref{erdos} for a slightly larger range. This improvement, however, was suggested to the author by Dr. Trudgian after the later computation had been done, and so we do not need to implement this. It would be a useful gambit to keep in mind for similar problems.

Thus, for fixed $n$ and $B$, it is easy to bound explicitly from above the number of solutions to (\ref{blah}). It remains to sum this bound over all valid values of $B$. However, we should note that given an integer $n$, there are not too many good choices of $B$, and this will allow us to make a further saving.

This comes from the observation that every prime $p >3$ satisfies $p^2 \equiv 1 \text{ mod } 24$. For with $p>3$ and $q>3$, it follows that (\ref{blah}) becomes
$$B \equiv n - 1 \text{ mod } 24,$$
and this confines $B$ to the integers in a single residue class modulo 24. 

Formally and explicitly, we argue as follows. Consider first the case where $B$ is an integer in the range
$$\frac{n-9}{A^2 n} \leq B < \frac{1}{A^2}.$$
The leftmost inequality above keeps $p \leq 3$. Here, there are clearly at most 
$$\frac{9}{A^2 n} + 1$$
integer values for $B$. We now consider the case where $p>3$, and it follows that $B \equiv n - 1 \text{ mod } 24$. Clearly, then, there are at most 
$$\frac{1}{24 A^2} + 1$$ values for $B$ in this range. Therefore, in total, there are at most 
$$2 + \frac{1}{24 A^2} + \frac{9}{A^2 n}$$
values of $B$ for which we need to sum the solution counts to (\ref{blah}). Also, we must also consider that $w=4$ for $B=1$. Therefore, we have that the number of solutions to (\ref{blah}) summed over $B$ is bounded above by
$$2^{\omega(n)-1}\bigg(3 + \frac{1}{24 A^2} + \frac{9}{A^2 n}\bigg).$$
Therefore, the number of primes $p$, each counted with weight $\log p$, which satisfy (\ref{blah}) is at most
\begin{equation} \label{cont4}
\sum_{A \sqrt{n} \leq q < \sqrt{n}} \sum_{\substack{p < \sqrt{n} \\ n \equiv p^2 \text{ mod } q^2}} \log p <  2^{1.3841 \log n / \log \log n}\bigg(\frac{3}{2} + \frac{1}{48 A^2} + \frac{9}{2A^2 n}\bigg) \log n.
\end{equation}

Now, collecting together (\ref{cont1}), (\ref{cont2}), (\ref{cont3}) and (\ref{cont4}), we have that the weighted count over all the so-called mischevious primes can be bounded hideously thus
\begin{eqnarray*}
\sum_{q < \sqrt{n}} \sum_{\substack{p < \sqrt{n} \\ n \equiv p^2 \text{ mod } q^2}} \log p & < & \bigg(0.568+\frac{0.00183}{\frac{1}{4}-c}+ (n^{-2c}+n^{-c})\log n \bigg) \sqrt{n} \\
& &+ \;\frac{A \sqrt{n} \log n}{\log(A \sqrt{n})}\bigg(1+\frac{1.2762}{\log (A \sqrt{n})}\bigg) \\
& &+ \;2^{1.3841 \log n / \log \log n}\bigg(\frac{3}{2} + \frac{1}{48 A^2} + \frac{9}{2A^2 n}\bigg) \log n.
\end{eqnarray*}
As expected, however, the weighted count over all primes exceeds this for large enough $n$ and good choices of $c$ and $A$. Dusart \cite{dusart} gives us that 
$$\theta(x) \geq x - 0.2 \frac{x}{\log^2 x}$$
for all $x \geq 3 594 641$, and thus it follows that
$$\theta(\sqrt{n}) \geq \sqrt{n} - 0.8 \frac{\sqrt{n}}{\log^2 n}$$
for all $n \geq 10^{14}$. There are sharper bounds that one could use here (see Trudgian \cite{trudgianpomerance} and Faber--Kadiri \cite{faberkadiri} for modern results), however Dusart's result is both easily stated and suitable for the application.

 Therefore, if we denote by $R(n)$ the (weighted) count of primes $p$ such that $n-p^2$ is square-free, it follows that
\begin{eqnarray*}
R(n) & > & \bigg(1-0.568-\frac{0.00183}{\frac{1}{4}-c}-\frac{0.8}{\log^2 n} -(n^{-2c}+n^{-c})\log n\bigg) \sqrt{n}\\
& - &\frac{A \sqrt{n} \log n}{\log(A \sqrt{n})}\bigg(1+\frac{1.2762}{\log (A \sqrt{n})}\bigg) \\
& - & 2^{1.3841 \log n / \log \log n}\bigg(\frac{3}{2} + \frac{1}{48 A^2} + \frac{9}{2A^2 n}\bigg) \log n.
\end{eqnarray*}
It is now straightforward to check that choosing $c=0.209$ and $A=0.0685$ gives $R(n)>0$ for all $n \geq 2.5 \times 10^{14}$, that is, Theorem \ref{erdos} is true for this range of integers.

We will now describe Platt's computation to confirm that all integers $n$ satisfying $10\leq n \leq 4\,000\,023\,301\,851\,135$ and $n\not\equiv 1 \mod 4$ can be written as the sum of the square of a prime and a square-free number. This is actually much further than we needed to check, but we did not expect our analytic approach to fare as well as it did. We will first describe the algorithm used, and then say a few words about its implementation.

We aim to test $3\cdot 10^{15}$ different values of $n$. We quickly conclude that we cannot afford to individually test candidate numbers of the type $n-p^2$ to see if they are square-free. We proceed instead by chosing a largest prime $P$ and a sieve width $W$. To check all the integers in $[N,N+W)$ we first sieve all the integers in $[N-P^2,N+W-4)$ by crossing out any that are divisible by a prime square $p^2$ with $p<\sqrt{(N+W-5)/2}$. Now, for each $n\in[N,N+W)$ such that $n\not\equiv 1 \mod 4$, we look in our sieve to see if $n-4$ is square-free\footnote{Unless $n\equiv 0 \mod 4$.}. If not, we try $n-9$, then $n-25$, and so on until $n-p^2$ is square-free. If it fails all these tests up to and including $n-P^2$, we output $n$ for later checking.

Numbers of this size fit comfortably in the $64$ bit native word size of modern CPUs and we implemented the algorithm in C++. We use a character array for the sieve and chose a sieve width $W=2^{31}$ as this allows us to run $16$ such sieves in parallel in the memory available. We set the prime limit $P=43$ as this was found to reduce the number of failures to a manageable level (see below). To generate the primes used to sieve the character array, we used Kim Walisch's \textit{primesieve} \cite{Walisch2012}.

We were able to run $16$ threads on a node of the University of Bristol's Bluecrystal cluster \cite{ACRC2014} and in total we required $5,400$ core hours of CPU time to check all $n\in[2048,4\,000\,023\,301\,851\,135]$. Here, $4\,915$ values of $n$ were rejected as none of $n-p^2$ with $p\leq 43$ were square-free. We checked these $4\,915$ cases in seconds using PARI \cite{Batut2000} and found that $p=47$ eliminated $4\,290$ of them, $p=53$ does this for a further $538$, $p=59$ for $14$ more, $p=61$ for $61$ more values of these $n$, $p=67$ does not help, $p=71$ knocks off $11$ more and the last one standing, $n=1\,623\,364\,493\,706\,484$ falls away with $p=73$. Finally, we use PARI again to check $n\in[10,2047]$ with $n\not\equiv 1 \mod 4$ and we are done.

It is interesting to consider the efficiency of the main part of this algorithm. The CPUs on the compute nodes of Phase III are $2.6$GHz Intel\textsuperscript{\textregistered} Xeon\textsuperscript{\textregistered} processors and we checked $3\cdot 10^{15}$ individual $n$ in $5\,400$ hours. This averages less than $17$ clock ticks per value of $n$ which suggests that the implementation must have made good use of cache. 

This computation completes the proof of Theorem \ref{erdos}. We note again that, for the interested reader, there are plenty more open problems of this flavour that may yield to similar techniques. We provide a collection of these in Chapter~6.


\chapter{Solving a Curious Inequality of Ramanujan}
\label{chapter5}

\begin{quote}
\textit{``Curiouser and curiouser!'' cried Alice (she was so much surprised that for the moment she quite forgot how to speak good English).}
\end{quote}

In this chapter, we detail another joint investigation, made by the author and Platt, on $\pi(x)$, the number of primes which are less than or equal to $x$. The work of this chapter has now been published \cite{dudekplattem} in \textit{Experimental Mathematics}.

In one of his notebooks, Ramanujan (see the preservations by Berndt \cite[Ch. 24]{berndt}) proved that the inequality
\begin{equation} \label{inequality}
\pi(x)^2 < \frac{e x}{\log x} \pi\Big(\frac{x}{e}\Big)
\end{equation}
holds for all sufficiently large values of $x$. Berndt \cite{berndt} states that Wheeler, Keiper and Galway used \textsc{Mathematica} in an attempt to determine an $x_0$ such that (\ref{inequality}) holds for all $x \geq x_0$. They were unsuccessful, but independently Galway was able to establish that the largest prime counterexample below $10^{11}$ occurs at $x = 38, 358, 837, 677$. 

Hassani \cite{Hassani} looked at the problem in $2012$ and established \textit{inter alia} the following theorem.
\begin{theorem}[Hassani]\label{theorem:hassani}
Assuming the Riemann Hypothesis, the inequality
$$\pi(x)^2 < \frac{e x}{\log x} \pi\Big(\frac{x}{e}\Big)$$
holds for all $x\geq 138,766,146,692,471,228$.
\end{theorem}

The first objective of this chapter is to provide an estimate for the inequality (\ref{inequality}) without the condition of the Riemann hypothesis. Using standard estimates on the error term in the Prime Number Theorem, we will prove the following theorem.
\begin{theorem}\label{theorem:uncon}
Without any condition, the inequality
$$\pi(x)^2 < \frac{e x}{\log x} \pi\Big(\frac{x}{e}\Big)$$
holds for all $x\geq\exp(9394)$.
\end{theorem}

We then solve the problem completely on the assumption of the Riemann hypothesis, showing that the counterexample found by Wheeler, Keiper and Galway is the largest.

\begin{theorem}\label{theorem:con}
Assuming the Riemann Hypothesis, the largest integer counterexample to
$$\pi(x)^2 < \frac{e x}{\log x} \pi\Big(\frac{x}{e}\Big)$$
is that at $x=38, 358, 837, 682$.
\end{theorem}

We will look at the unconditional result first.

\section{The unconditional result}

We start by giving Ramanujan's original and, we think, rather fetching proof, which is based on de la Vall\'ee Poussin's rendition of the Prime Number Theorem, or more specifically that
\begin{equation} \label{piexpansion}
\pi(x) = x \sum_{k=0}^{4} \frac{k!}{\log^{k+1}x}+O\Big(\frac{x}{\log^6 x}\Big)
\end{equation}
as $x \rightarrow \infty$. As such we have the two estimates
\begin{equation*} \label{pisquared}
\pi^2 (x) = x^2 \Big\{ \frac{1}{\log^2 x} + \frac{2}{\log^3 x} + \frac{5}{\log^4 x} + \frac{16}{\log^5 x} + \frac{64}{\log^6 x} \Big\} + O\Big( \frac{x^2}{\log^7 x}\Big)
\end{equation*}
and
\begin{eqnarray*} \label{pimult}
\frac{e x}{\log x} \pi \Big(\frac{x}{e}\Big) & = & \frac{x^2}{\log x} \Big\{ \sum_{k=0}^{4} \frac{k!}{(\log x-1)^{k+1}} \Big\} + O\Big( \frac{x^2}{\log^7 x} \Big)  \\
& = & x^2 \Big\{ \frac{1}{\log^2 x} + \frac{2}{\log^3 x} + \frac{5}{\log^4 x} + \frac{16}{\log^5 x} + \frac{65}{\log^6 x} \Big\} + O\Big( \frac{x^2}{\log^7 x}\Big).
\end{eqnarray*}
Subtracting the above two expressions gives
\begin{equation} \label{makemenegative}
\pi^2 (x) - \frac{e x}{\log x} \pi \Big(\frac{x}{e}\Big) = - \frac{x^2}{\log^6 x} + O \Big(\frac{x^2}{\log^7 x} \Big)
\end{equation}
which is negative for sufficiently large values of $x$. This completes the proof.

The proof itself should serve as a tribute to the workings of Ramanujan's mind, for surely one would not calculate the asymptotic expansions of such functions without the inkling that doing so would be fruitful. 

Note that if one were to work through the above proof using explicit estimates on the asymptotic expansion of the prime-counting function, then one would be able to make precise what is meant by ``sufficiently large''. The following lemma shows how we do this.

\begin{lemma} \label{lem1}
Let $m_a, M_a \in \mathbb{R}$  and suppose that for $x>x_a$ we have
$$ x \sum_{k=0}^{4} \frac{k!}{\log^{k+1}x}+ \frac{m_a x}{\log^6 x} < \pi(x) < x \sum_{k=0}^{4} \frac{k!}{\log^{k+1}x}+\frac{M_a x}{\log^6 x}.$$
Then Ramanujan's inequality is true if 
$$x > \max( e x_{a},x_{a}' )$$
where a value for $x_{a}'$ can be obtained in the proof and is completely determined by $m_a, M_a$ and $x_{a}$.

\end{lemma}

\begin{proof}
Following along the lines of Ramanujan's proof we have for $x > x_{a}$
\begin{equation} \label{pipi}
\pi^2(x)  <  x^2 \Big\{ \frac{1}{\log^2 x}+ \frac{2}{\log^3 x}+ \frac{5}{\log^4 x}+ \frac{16}{\log^5 x}+ \frac{64}{\log^6 x} + \frac{\epsilon_{M_a}(x)}{\log^7 x} \Big\}, 
\end{equation}
where 
$$\epsilon_{M_a} (x) = 72 + 2 M_a + \frac{2M_a+132}{\log x} + \frac{4M_a+288}{\log^2 x} + \frac{12 M_a+576}{\log^3 x}+\frac{48M_a}{\log^4 x} + \frac{M_a^2}{\log^5 x}.$$

The other term requires slightly more trickery; we have for $x > e x_{a}$
$$\frac{ex}{\log x} \pi \Big(\frac{x}{e} \Big) > \frac{x^2}{\log x} \Big( \sum_{k=0}^{4} \frac{k!}{(\log x - 1)^{k+1}}\Big) + \frac{m_a x}{(\log x-1)^{6}}. $$
We make use of the inequality
\begin{eqnarray*} 
\frac{1}{(\log x - 1)^{k+1}} & = & \frac{1}{\log^{k+1} x} \Big(1+ \frac{1}{\log x} + \frac{1}{\log^2 x} + \frac{1}{\log^3 x} + \cdots \Big)^{k+1} \\ \\
& > & \frac{1}{\log^{k+1} x} \Big(1+ \frac{1}{\log x}+ \cdots + \frac{1}{\log^{5-k} x} \Big)^{k+1}
\end{eqnarray*}
to get
\begin{equation} \label{epi}
\frac{ex}{\log x} \pi \Big(\frac{x}{e} \Big) > x^2 \Big\{ \frac{1}{\log^2 x}+ \frac{2}{\log^3 x}+ \frac{5}{\log^4 x}+ \frac{16}{\log^5 x}+ \frac{64}{\log^6 x} + \frac{\epsilon_{m_a}(x)}{\log^7 x} \Big\}, 
\end{equation}
where
$$\epsilon_{m_a}(x) = 206+m_a+\frac{364}{\log x} + \frac{381}{\log^2 x}+\frac{238}{\log^3 x} + \frac{97}{\log^4 x} + \frac{30}{\log^5 x} + \frac{8}{\log^6 x}.$$
Now, subtracting $(\ref{epi})$ from $(\ref{pipi})$ we have
$$\pi^2(x) - \frac{ex}{\log x} \pi \Big( \frac{x}{e} \Big) < \frac{x^2}{\log^6 x} \Big(-1 + \frac{\epsilon_{M_a} (x) - \epsilon_{m_a} (x)}{\log x} \Big).$$
The right hand side is negative if
$$\log x > \epsilon_{M_a} (x) - \epsilon_{m_a} (x),$$
and so we can then choose $x_a'$ to be some value $x$ which satisfies this.
\end{proof}

The aim is to reduce $\max(ex_{a},x_{a}' )$ so as to get the sharpest bound available using this method and modern estimates involving the prime counting function.  We thus look at developing the explicit bounds on $\pi(x)$ that are required to invoke Lemma \ref{lem1}. We need to call on Corollary 1 of Mossinghoff and Trudgian \cite{trudgianmossinghoff}, which bounds the error in approximating the Chebyshev $\theta$-function with $x$.
\begin{lemma} \label{explicitPNT}
Let
$$\epsilon_0 (x) = \sqrt{\frac{8}{17 \pi}} X^{1/2} e^{-X}, \hspace{0.2in} X= \sqrt{(\log x)/R}, \hspace{0.2in}  R = 6.315.$$
Then
$$|\theta(x) - x | \leq x \epsilon_0(x), \hspace{0.2in} x \geq 149$$
\end{lemma}

This is another form of the Prime Number Theorem, and is able to give us the estimates required to use Lemma \ref{lem1}. For any choice of $a>0$, it is possible to use the Lemma \ref{explicitPNT} to find some $x_a >0$ such that
\begin{equation} \label{chebyshevbound}
| \theta(x) - x | < a \frac{x}{\log^5 x}
\end{equation}
for all $x>x_a$; we simply need to find the range of $x$ for which 
$$\sqrt{\frac{8}{17 \pi}} \bigg( \frac{\log x}{R} \bigg)^{1/4} e^{-\sqrt{(\log x)/R}} <  \frac{a}{\log^5 x}. $$
As this may yield large values of $x_a$, we write $x = e^y$ (also $x_a = e^{y_a}$) and take logarithms to get the equivalent inequality
\begin{equation} \label{solve}
\log \bigg( \frac{1}{R^{1/4} a} \sqrt{\frac{8}{17 \pi}}  \bigg) + \frac{21}{4} \log y \leq \sqrt{\frac{y}{R}}.
\end{equation}

Now, suppose that, for any $a>0$ and some corresponding $x_a>0$ we have
$$\theta(x) < x + a \frac{x}{\log^5 x}$$
for all $x > x_a$. The technique of partial summation gives us that
\begin{eqnarray*}
\pi(x) & < & \frac{x}{\log x} + \int_2^x \frac{dt}{\log^2 t} + a \frac{x}{\log^6 x} + a \int_2^x \frac{dt}{\log^7 t} \\ \\
& < & x \Big( \sum_{k=0}^{4} \frac{k!}{\log^{k+1} x} \Big) + (120+a) \frac{x}{\log^6 x} + (720 +a) \int_2^x \frac{dt}{\log^7 t}.
\end{eqnarray*}
We can estimate the remaining integral here by
\begin{eqnarray*}
\int_2^x \frac{dt}{\log^7 t} & < & \frac{x}{\log^7 x} + 7 \int_2^x \frac{dt}{\log^8 t} \\ \\
& < & \frac{x}{\log^7 x} + 7 \bigg( \int_2^{\sqrt{x}} \frac{dt}{\log^8 t} +  \int_{\sqrt{x}}^{x} \frac{dt}{\log^8 t}\bigg) \\ \\
& < &  \frac{x}{\log^7 x} + 7 \Big( \frac{\sqrt{x}}{\log^8 2} + \frac{2^8 x}{\log^8 x} \Big).
\end{eqnarray*}
Putting it all together we have that
$$\pi(x) < x \Big( \sum_{k=0}^{4} \frac{k!}{\log^{k+1} x} \Big) + M_a \frac{x}{\log^6 x}$$
for all $x>x_a$, where 
\begin{equation} \label{M}
M_a = 120 + a +\frac{a+720}{\log x_a} + \frac{1792 a + 1290240}{\log^2 x_a} + \Big( \frac{5040+7a}{\log^8 2} \Big) \frac{\log^6 x_a}{\sqrt{x_a}}.
\end{equation}

In an almost identical way, we can obtain for $x>x_a$ that
$$\pi(x) > x \Big( \sum_{k=0}^{4} \frac{k!}{\log^{k+1} x} \Big) + m_a \frac{x}{\log^6 x}$$
where
\begin{equation} \label{m}
m_a = 120 -a - \frac{a}{\log x_a} - \frac{1792}{\log^2 x_a} - 2 A \frac{\log^6 x_a}{x_a} - \frac{7 a \log^6 x_a}{\log^8 2 \sqrt{x_a}}
\end{equation}
and
$$A = \sum_{k=1}^{5} \frac{k!}{\log^{k+1} 2} < 1266.1.$$

Now that we have our estimates, we can launch directly into the proof of Theorem \ref{theorem:uncon}. Our method is as follows. We choose some $a>0$ such that we wish for
$$|\theta(x) - x| < a \frac{x}{\log^5 x} $$
to hold for $x>x_a = e^{y_a}$. We simply plug our desired value of $a$ into (\ref{solve}) and use \textsc{Mathematica} to search for some value of $y_a$, such that the inequality holds for all $x > e^{y_a}$.  We then use (\ref{M}) and (\ref{m}) to calculate two values $m_a$ and $M_a$ such that
$$ x \sum_{k=0}^{4} \frac{k!}{\log^{k+1}x}+ \frac{m_a x}{\log^6 x} < \pi(x) < x \sum_{k=0}^{4} \frac{k!}{\log^{k+1}x}+\frac{M_a x}{\log^6 x}$$
holds for $x>e^{y_a}$. Then, by Lemma \ref{lem1}, we find some value $x_{a}' = e^{y_a'}$ (dependent on $a$, $m_a$ and $M_a$, and thus really only on $a$) such that Ramanujan's inequality is true for $x > \max(ex_a, x_a')$. 

One finds that small values of $a$, give rise to large values of $x_a$, yet small values of $x_a'$. Similiarly, large values of $a$ will yield small $x_a$ yet large values of $x_a '$. Of course, we want $x_a$ and $x_a'$ to be comparable, so that we might lower their maximum as much as possible. Thus, the idea is to select $a$ so that $ex_a$ and $x_a'$ are as close as possible.

It follows immediately from the above and Lemma \ref{lem1}, upon choosing $a=3130$, that $x_a = \exp(9393)$ and $x_a' = \exp(9394)$ are suitable values. This gives us Theorem \ref{theorem:uncon}.

\section{Estimates on the Riemann hypothesis}

In this section, we prove Theorem \ref{theorem:con}. We assume the Riemann Hypothesis and can therefore rely on Schoenfeld's \cite{schoenfeld} conditional bound for the prime counting function.
\begin{theorem}\label{theorem:schoen}
Assume the Riemann hypothesis. For $x\geq 2657$ we have
\begin{equation*} \label{schoen}
|\pi(x) - \emph{li}(x) | < \frac{1}{8 \pi} \sqrt{x} \log x.
\end{equation*}
\end{theorem}

We now aim to improve on Theorem \ref{theorem:hassani} to the extent that a numerical computation to check the remaining cases become feasible. The following result will bridge the gap by showing that Ramanujan's inequality (\ref{inequality}) also holds in the range $1.15 \cdot 10^{16} \leq x \leq 1.39 \cdot 10^{17}$.
\begin{lemma}\label{lem:lowlim}
Assuming the Riemann Hypothesis, we have
$$\pi^2(x)<\frac{ex}{\log x}\pi\left(\frac{x}{e}\right)$$
for all $x\geq 1.15\cdot 10^{16}$.
\end{lemma}
\begin{proof}
Due to the work of Hassani, we really only need to prove this for the range $1.15 \cdot 10^{16} \leq x \leq 1.39 \cdot 10^{17}$. Platt and Trudgian \cite{platttrudgian} have recently confirmed that $\pi(x)<\text{li}(x)$ holds for $x \leq 1.39 \cdot10^{17}$. Together with Theorem \ref{theorem:schoen} we see that
$$f(x)=\pi^2(x) - \frac{ex}{\log x} \pi \Big( \frac{x}{e} \Big)$$
is bounded above by
$$g(x) =  \text{li}^2(x) - \frac{ex}{\log x} \bigg( \text{li}\Big(\frac{x}{e} \Big) - \frac{1}{8 \pi} \sqrt{\frac{x}{e}} (\log x - 1) \bigg) $$
for all $x \geq 2657 e$. Berndt \cite[Pg.\ 114]{berndt} uses some elementary calculus to show that a similar function to the above is monotonically increasing over some range. One can use that same technique here to show that $g(x)$ is monotonically decreasing for all $x \geq 10^{16}$, as the derivative of $g(x)$ is straightforward to compute. Then, \textsc{Mathematica} can be used to show that
$$g(1.15 \cdot 10^{16}) < -3.2 \cdot 10^{19} <0$$
and thus $g(x)$ is negative for all $x \geq 1.15 \cdot 10^{16}$. 
\end{proof}

Finally, we wish to show by computation that there are no counterexamples to (\ref{inequality}) in the interval $[10^{11}, 1.15 \cdot 10^{16}]$. As before, we write
$$f(x) = \pi^2(x) - \frac{ex}{\log x} \pi\bigg( \frac{x}{e} \bigg).$$
Note that $f$ is strictly decreasing between primes, so we could simply check that $f(p)<0$ for all primes $p$ in the required range. However, there are roughly $3.2\cdot 10^{14}$ primes to consider\footnote{Or precisely $319,870,505,122,591$.} and this many evaluations of $f$ would be computationally too expensive. Instead, we employ a simple stepping argument.
\begin{lemma}\label{lem:step}
Let $x_0$ be in the interval $[10^{11},1.15 \cdot 10^{16}]$ with $f(x_0) < 0$. Set $\epsilon=\sqrt{\pi^2(x_0)-f(x_0)}-\pi(x_0)$. Then $f(x) < 0$ for all $x\in[x_0, x_0 + \epsilon]$.
\end{lemma}
\begin{proof}
We have for $x_0\geq e$ and $\epsilon>0$
\begin{eqnarray*}
f(x_0+\epsilon) & = & \pi^2(x_0+\epsilon)-\frac{e(x_0+\epsilon)}{\log(x_0+\epsilon)}\pi\left(\frac{x_0+\epsilon}{e}\right)\\
& \leq & \pi^2(x_0)+2\pi(x_0)\epsilon+\epsilon^2-\frac{ex_0}{\log x_0}\pi\left(\frac{x_0}{e}\right)\\
& = & f(x_0)+2\pi(x_0)\epsilon+\epsilon^2.
\end{eqnarray*}
Setting $f(x_0+\epsilon)=0$ and solving the resulting quadratic in $\epsilon$ gives us our lemma.
\end{proof}

Suppose we have access to a table containing values of $\pi(x_i)$ with $x_{i+1}>x_i$ for all $i$. Then we can compute an interval containing $\pi(x)$ simply by looking up $\pi(x_i)$ and $\pi(x_{i+1})$ where $x_i\leq x \leq x_{i+1}$. Repeating this for $x/e$ we can determine an interval $[a,b]$ for which $f(x)$ must lie. Assuming $b$ is negative, we can use Lemma \ref{lem:step} to step to the next value of $x$ and repeat.

Oliviera e Silva has produced extensive tables of $\pi(x)$ \cite{Oliviera2012}. Unfortunately, these are not of a sufficiently fine granularity to support the algorithm outlined above. In other words the estimates on $\pi(x)$ and $\pi(x/e)$ we can derive from these tables alone are too imprecise and do not determine the sign of $f(x)$ uniquely. We looked at the possibility of refining the coarse intervals provided by these tables using Montgomery and Vaughan's explicit version of the Brun--Titchmarsh \cite{MV} theorem but to no avail. Instead, we (or, more specifically, Dave Platt) re-sieved the range $[1\cdot 10^{10},1.15\cdot 10^{16}]$ to produce exact values for $\pi(x_i)$ where the $x_i$ were more closely spaced. Table \ref{tab:pixs} provides the details.\footnote{At the same time we double checked Oliviera e Silva's computations and, as expected, we found no discrepancies.}

\begin{table}[ht] 
\caption{The Prime Sieving Parameters} 
\label{tab:pixs} 
\centering 
\begin{tabular}{c c c} 
\hline\hline 
From & To & Spacing\\[0.5 ex] \hline
$10^{10}$ & $10^{11}$ & $10^3$\\  
$10^{11}$ & $10^{12}$ & $10^4$\\  
$10^{12}$ & $10^{13}$ & $10^5$\\  
$10^{13}$ & $10^{14}$ & $10^6$\\  
$10^{14}$ & $10^{15}$ & $10^7$\\  
$10^{15}$ & $10^{16}$ & $10^8$\\  
$10^{16}$ & $1.15\cdot 10^{16}$ & $10^9$\\  
\hline\hline 
\end{tabular} 
\end{table} 

We used Kim Walisch's \textit{primesieve} package \cite{Walisch2012} to perform the sieving and it required a total of about $300$ hours running on nodes of the University of Bristol's Bluecrystal cluster \cite{ACRC2014}.\footnote{Each node comprises two $8$ core Intel\textsuperscript{\circledR} Xeon\textsuperscript{\circledR} E5-2670 CPUs running at $2.6$ GHz and we ran with one thread per core.}

Using Lemma \ref{lem:step} with these tables, actually confirming that $f(x)<0$ for all $x\in[1\cdot 10^{11},1.15\cdot 10^{16}]$ took less than $5$ minutes on a single core. We had to step (and therefore compute $f(x)$) about $5.3\cdot 10^8$ times to span this range. We sampled at every $100,000$th step and Figure \ref{fig:plot} shows a log/log plot of $x$ against $-f(x)$ for these samples.\footnote{Actually, we use the midpoint of the interval computed for $-f(x)$.}

\begin{figure}[tbp]
\centering
\fbox{\includegraphics[width=0.45\linewidth,angle=270]{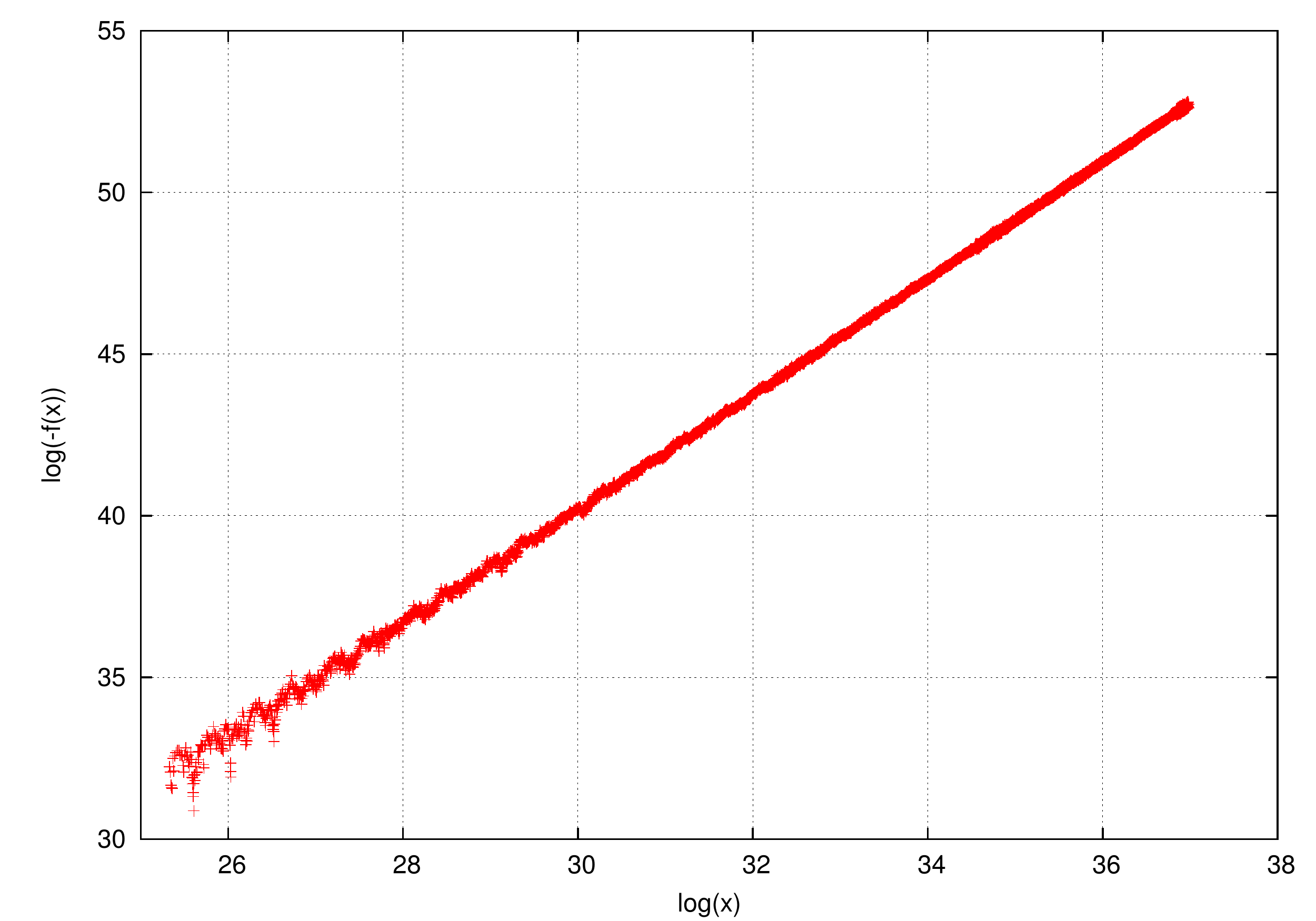}}
\caption{$\log(x)$ vs. $\log(-f(x))$.}
\label{fig:plot}
\end{figure}

No counterexamples to (\ref{inequality}) where uncovered by this computation and so we can now state that there are no counterexamples in the interval $[10^{11}, 1.15 \cdot 10^{16}]$. This concludes the proof of Theorem \ref{theorem:con}, and so ends this chapter of work. The next chapter contains a good assortment of open problems for the reader to investigate, including a suggested generalisation of the work contained in this chapter.


\chapter{An Offering of Open Problems}
\label{chapter6}

\begin{quote}
\textit{``I almost wish I hadn't gone down that rabbit-hole -- and yet -- and yet -- it's rather curious, you know, this sort of life!''}
\end{quote}

Whilst solving the problems that made this thesis, the author came across a suite of related problems. Within this chapter, we offer these to the avid reader in the hope that interest in the area of explicit methods in number theory burgeons.

Certainly, one should adopt the practice of searching for new problems every once in a while. Besides the standard citation searches offered by the likes of MathSciNet and Google Scholar, there is a lot to be gained by sifting through the information on Wikipedia, Wolfram MathWorld, StackExchange and MathOverflow. Of course, these latter sources offer not peer-reviewed research, but a community dialogue for one to look through, and it is surprising how much can be gained from this.

For example, the inequality of Ramanujan studied in Chapter 5 was first found on the Wolfram MathWorld page \cite{mathworld} for the prime counting function. And similarly, it would be unwise for one working on the Riemann hypothesis to avoid the Wikipedia page \cite{wiki} on this topic for reasons of scholarly lewdness. 

The author also found Guy's book of unsolved problems \cite{guy}, Murty's book of exercises \cite{murty} and the Mitrinovi\'{c}--S\'{a}ndor--Crstici handbook \cite{sandor} to be excellent companions to the working analytic number theorist.

\begin{problem}
Consider the explicit version of the truncated Riemann--von Mangoldt explicit formula developed in Chapter 2. Wolke \cite{wolke} has established a (not explicit) version of this formula with an error term that is
$$O\bigg(  \frac{x}{T} \frac{ \log x}{\log (x/T)} \bigg).$$
It would be interesting to make this explicit. 
\end{problem}

\begin{problem}
Using an explicit version of Wolke's formula, one could estimate the sum over the zeroes using Ramar\'{e}'s explicit zero-density estimate \cite{ramare} and the Mossinghoff--Trudgian zero-free region \cite{trudgianmossinghoff}. This would give an explicit bound for $|\psi(x)-x|$ that might improve that of Mossinghoff and Trudgian.
\end{problem}

\begin{problem} \label{B}
In his paper, Wolke \cite{wolke} shows how his explicit formula can be used to prove Cram\'{e}r's theorem that there is a prime in the interval $(x-c \sqrt{x} \log x, x)$ for some $c>0$ and sufficiently large $x$. It would be interesting to see if Wolke's method, made explicit, would yield better estimates for $c$ than those given in Chapter 3 of this thesis.
\end{problem}

\begin{problem}
It would be interesting to see the explicit version of the truncated Riemann--von Mangoldt explicit formula proven in a faster way. Specifically, can one write the sum over the zeroes as
$$\sum_{\rho} \frac{x^{\rho}}{\rho} = \sum_{|\gamma| \leq T} \frac{x^{\rho}}{\rho} + \sum_{|\gamma| > T} \frac{x^{\rho}}{\rho} $$
and directly bound the rightmost sum?
\end{problem}

\begin{problem}
The Riemann--von Mangoldt explicit formula has been used in various applications. It would be interesting to see which of these applications can now be reworked explicitly using Theorem \ref{explicitformula}.
\end{problem}

\begin{problem}
It would be interesting to see if other explicit formulas can produce better bounds on primes between cubes. The author attempted an investigation of this using $\psi_1(x)$, but was not able to get primes between cubes using Ford's zero-free region and Ramar\'{e}'s zero-density estimate (one can, however, still get short interval results).
\end{problem}

\begin{problem}
Although proving that there is a prime between any two consecutive cubes seems out of reach (at least for this author), one could try to prove the existence of almost-primes between cubes. There are some details on this in Ivi\'{c}'s book \cite[Ch.\ 12.7]{ivicbook}.
\end{problem}

\begin{problem} \label{additive}
In Chapter 4, we make two theorems in additive number theory completely explicit. There are many similar problems that one could look at. For example, fix $a, q \in \mathbb{N}$ such that $(a,q)=1$. One could attempt to show that every large enough positive integer $n$ could be written as
$$n = p + m$$
where $p$ is a prime such that $p \equiv a \mod q$ and $m$ is square-free.
\end{problem}

\begin{problem}
One could prove that every positive integer is the \textit{difference} of a prime and a square-free number. 
\end{problem}

\begin{problem}
Given some positive integer $K$, explicitly determine a constant $C(k)$ such that every even integer $n>C(k)$ can be written in the form
$$n = p_1 + p_2 + 2^{a_1} + 2^{a_2} + \cdots + 2^{a_r}$$ 
where $p_1$ and $p_2$ are primes, $a_i$ is a positive integer for all $i$ and $r \leq K$.
\end{problem}

\begin{problem}
In Erd\H{o}s' paper \cite{erdos}, he also shows that, in the case where $n\equiv 1 \mod 4$ and is sufficiently large, $n$ may be written in the form $n = 4p^2+f$ where $p$ is a prime and $f$ is square-free. I am not sure why Dave Platt and I walked around this problem, but one could surely resolve it similarly to the case where $n \not \equiv 1 \mod 4$.
\end{problem}

\begin{problem}
Prove that every integer (in the appropriate congruence classes and from some number onwards) can be written as the sum of a $k$-th power of a prime and an $l$-free number (see Rao \cite{rao} for a proof that every sufficiently large integer can be so written).
\end{problem}

\begin{problem} 
In Chapter 4, we showed how Platt's \cite{Platt2013} work on verifying GRH up to certain heights could be used to improve the bounds on the Prime Number Theorem in arithmetic progressions. It would be very useful if one were to rework the entire paper of Ramar\'{e} and Rumely \cite{ramarerumely} using this verification.
\end{problem}

\begin{problem} 
It would be interesting to see if the inequality of Ramanujan that we studied in Chapter 5 could be suitably generalised. For example, Ramanujan considered two expressions involving $\pi(x)$ that agreed for the first four terms of their series expansions but not thereafter. Given some integer $m$, would it be possible to contrive two expressions involving $\pi(x)$ that agree for the first $m$ terms but not thereafter? 

Note that Hassani \cite{hassanigen} has generalised the problem in a different way, which might also be of interest.
\end{problem}

The above problems are few in a field of many, but these are certainly the ones of most interest to the author.


\newpage
  \appendix



  \addcontentsline{toc}{chapter}{Bibliography}


\nocite{*}

\bibliographystyle{plain}

\bibliography{biblio}

\end{document}